\documentclass{amsart}

\usepackage{amsmath,amssymb,amsfonts}
\usepackage[english]{babel}
\usepackage{enumitem}
\usepackage{amsxtra, mathtools,stmaryrd}
\usepackage[all]{xy}
\usepackage[colorlinks=true, pagebackref]{hyperref}
\usepackage{easyReview}
\usepackage[color, leftbars]{changebar}

\newtheorem{teo}{Theorem}[section]
\newtheorem{prop}[teo]{Proposition}
\newtheorem{lemma}[teo]{Lemma}
\newtheorem{cor}[teo]{Corollary}
\theoremstyle{definition}
\newtheorem{defi}[teo]{Definition}
\newtheorem{rem}[teo]{Remark}
\newtheorem{example}[teo]{Example}

\newcommand{\Ext}{\mathchoice{{\textstyle\bigwedge}}
	{{\bigwedge}}
	{{\textstyle\wedge}}
	{{\scriptstyle\wedge}}}

\newcommand{\lra}{\longrightarrow}

\newcommand{\ie}{\textit{i.e.\ }}

\newcommand{\resp}{\textit{resp. }}

\newcommand{\reg}{{\rm reg}}

\newcommand{\inv}{{\rm inv}}
\newcommand{\diag}{{\rm diag}}

\newcommand{\IP}{\mathbb{P}}
\newcommand{\IC}{\mathbb{C}}

\newcommand{\IZ}{\mathbb{Z}}

\newcommand{\IN}{\mathbb{N}}

\newcommand{\cF}{\mathcal{F}}

\newcommand{\cK}{\mathcal{K}}

\newcommand{\cO}{\mathcal{O}}
\newcommand{\cP}{\mathcal{P}}

\newcommand{\Aut}{\mathrm{Aut}}
\newcommand{\Alb}{\mathrm{Alb}}
\newcommand{\NS}{\mathrm{NS}}
\newcommand{\GL}{\mathrm{GL}}
\newcommand{\SU}{\mathrm{SU}}
\newcommand{\Sp}{\mathrm{Sp}}

\newcommand{\Spec}{\mathrm{Spec}}

\newcommand{\Pic}{\mathrm{Pic}}
\newcommand{\Cl}{\mathrm{Cl}}

\newcommand{\PP}{\mathrm{P}}

\newcommand{\id}{\mathrm{id}}
\newcommand{\age}{\mathrm{age}}
\newcommand{\ii}{\mathrm{i}}

\newcommand{\Hilb}[2]{{#2}^{[#1]}}
\newcommand{\Sym}[2]{{#2}^{(#1)}}
\newcommand{\Kum}[2]{{\rm K}_{#1}(#2)}

\cbcolor{blue}

\title[Log-Enriques varieties]{Logarithmic Enriques varieties}

\author{Samuel Boissi\`ere}
\address{Samuel Boissi\`ere,
	Laboratoire de Math\'ematiques et Applications,
	UMR 7348 du CNRS,
	B\^atiment H3,
	Boulevard Marie et Pierre Curie,
	Site du Futuroscope,
	TSA 61125,
	86073 Poitiers Cedex 9,
	France}
\email{samuel.boissiere@univ-poitiers.fr}
\urladdr{https://sboissie.pages.math.cnrs.fr/home/}

\author{Chiara Camere}
\address{Chiara Camere, Universit\`a degli Studi di Milano,
	Dipartimento di Matematica,
	Via Cesare Saldini 50,
	20133 Milano, Italy} 
\email{chiara.camere@unimi.it}
\urladdr{https://sites.unimi.it/camere/en/index.html}

\author{Alessandra Sarti}
\address{Alessandra Sarti,
	Laboratoire de Math\'ematiques et Applications,
	UMR 7348 du CNRS,
	B\^atiment H3,
	Boulevard Marie et Pierre Curie,
	Site du Futuroscope,
	TSA 61125,
	86073 Poitiers Cedex 9,
	France}
\email{alessandra.sarti@univ-poitiers.fr}
\urladdr{https://sarti.pages.math.cnrs.fr/home/}

\date{\today}

\begin{document}

	

	\begin{abstract}
		We introduce logarithmic Enriques varieties as a singular analogue of Enriques manifolds, generalizing the notion of log-Enriques surfaces introduced by Zhang. We focus mainly on the properties of the subfamily of logarithmic Enriques varieties that admit a quasi-\'etale cover by a singular symplectic variety and we give many examples.
	\end{abstract}

	\maketitle

	\section{Introduction}
	
	In the Enriques--Kodaira classification of smooth complex algebraic surfaces, an \emph{Enriques surface} is by definition a minimal surface $X$ of Kodaira dimension $\kappa = 0$, arithmetic genus $p_g(X) \coloneqq h^0(K_X) = 0$ and irregularity $q(X) \coloneqq h^1(\cO_X) = 0$. It follows that the canonical divisor $K_X$ is $2$-torsion and that the \'etale double cover defined by the $2$-torsion line bundle $\cO_X(K_X)$ realizes $X$ as the quotient of a K3 surface $Y$ by a fixed point free nonsymplectic involution (see for instance~\cite[Chapter~VIII]{Beauville_book}). As a consequence, its fundamental group $\pi_1(X)$ is isomorphic to $\IZ/2\IZ$. The existence of a nonsymplectic automorphism imposes that $Y$ is projective, hence $X$ too (see for instance~\cite[Proposition~6]{Beauville_remarks}).
	Conversely, every fixed point free involution on a K3 
	surface $Y$ is nonsymplectic by the holomorphic Lefschetz fixed point formula. This imposes that $Y$ is projective, and the quotient variety is an Enriques surface. 
	
	In a different flavour, starting from a smooth complex algebraic surface $X$ with nontrivial but $2$-torsion canonical bundle, the \'etale degree two cover $Y\to X$ defined by the canonical sheaf is such that $K_Y = 0$, so $p_g(Y) = 1$  and $Y$ is a minimal surface which is either a K3 surface or an Abelian 
	surface.  If we further assume that $q(X) = 0$, then the Euler characteristic of $X$ is $\chi(X, \cO_X) = 1$, so $\chi(Y, \cO_Y)  = 2$, hence $q(Y) = 0$ and $Y$ is a K3 surface.
	We refer to Beauville~\cite[Chapter~VIII]{Beauville_book}, Barth--Hulek--Peters--van de Ven~\cite{BHPV} or Cossec--Dolgachev~\cite{CossecDolgachev} for more characterizations of Enriques surfaces.
	
	These equivalent points of view on Enriques surfaces offer different strategies to extend their definition to higher dimension in the nonsingular setup, which have been explored in~Boissi\`ere--Nieper-Wisskirchen--Sarti~\cite{BNWS}, Oguiso--Schr\"oer~\cite{OS1} and Yoshikawa~\cite{Yoshikawa}: these are called {\it Enriques manifolds}. A recent interest for Enriques manifolds occurs in the framework of the investigation of the Morrison--Kawamata cone conjecture by Pacienza and the third author~\cite{PS} and more recently by Gachet~\cite{Gachet}.  In light of recent developments on the Beauville--Bogomolov decomposition of singular varieties with trivial first Chern class~\cite{Druel, GKP_singular, HoringPeternell}, it is interesting to extend further the definition of Enriques manifolds to the singular setting.  
	
	In this paper, after reviewing the theory of singular symplectic varieties and of singular Calabi--Yau varieties in Section \ref{s:prelim}, and after discussing some properties of $K$-torsion varieties in Section \ref{s:torsion_K_trivial}, we first propose in Definition~\ref{def:enriques} a general definition of Enriques manifolds that interpolates with the previous ones cited above and then we propose in Definition~\ref{def:log_enriques} an extension of this notion to the singular setup, which we call \emph{logarithmic Enriques varieties}, (log-Enriques for short) following the term \emph{logarithmic Enriques surfaces} introduced by Zhang~\cite{Zhang}. Then in Section \ref{s: symplectic type}  we discuss general properties and characterizations of these varieties in the special case of logarithmic Enriques varieties \emph{of symplectic type}, \ie  those logarithmic Enriques varieties admitting a cyclic quasi-\'etale cover by a symplectic variety. In particular, in Proposition \ref{prop:conseq2} we show that if the quasi-\'etale cover $Y$ is a PSV (respectively an ISV) then the logarithmic Enriques variety $X$ is the quotient of $Y$ by the action of a purely nonsymplectic automorphism.	
	Finally, in Section \ref{s:examples_nonsingular} we give many examples: in particular, in the case of Enriques manifolds, we review the examples already mentioned in the cited references and we also construct some new examples of log-Enriques varieties. 
	
	Angel David R\'ios Ortiz, Francesco Denisi, Nikolaos Tsakanikas and Zhixin Xie~\cite{DRTX} have  been working independently and simultaneously on a variant of this notion that they called \emph{primitive Enriques varieties}, which is compatible with our definitions, although they explore different developments as ours.
	
	The authors warmly thank Francesco Denisi, Christian Lehn, Martina Monti, Gianluca Pacienza, Arvid Perego, Giulia Saccà, Roberto Svaldi and Nikolaos Tsakanikas for helpful discussions. The first and the  third author have been partially funded by the ANR/DFG project ANR-23-CE40-0026 ``Positivity on K-trivial varieties". The second author received partial support from grants Prin Project 2020 "Curves, Ricci flat Varieties and their Interactions" and Prin Project 2022 "Symplectic varieties: their interplay with Fano manifolds and derived categories"; she is a member of the Indam group GNSAGA and was supported by the International Research Laboratory (IRL)  LYSM "Ypatia Laboratory for Mathematical Sciences".
	We  thank the anonymous referees for their careful reading, corrections and suggestions of improvement.
	
	\section{Preliminaries}
	\label{s:prelim}

	In this paper, the term ``variety'' denotes an integral separated noetherian scheme of finite type over the field of complex numbers.
	Let $X$ be a projective variety.
	We denote by $X_\reg$ its regular locus with its  open embedding $\iota\colon X_\reg \hookrightarrow X$, by $\Omega^1_{X_\reg}$  the sheaf of K\"ahler differentials and by $\Omega^p_{X_\reg}\coloneqq\Ext^p \Omega^1_{X_\reg}$ the sheaf of $p$-forms. For any $p\in \IN$ such that $0\leq p\leq \dim X$, we define:
	\[
	\Omega^{[p]}_X\coloneqq \iota_\ast  \Omega^p_{X_\reg}.
	\]
	Assuming that $X$ is normal, $\Omega^{[p]}_X$ is a reflexive sheaf isomorphic to the bidual $\left(\Omega_X^p\right)^{\ast\ast}$, whose sections are called \emph{reflexive forms} (see for instance~\cite{GKKP, Hartshorne, Reid_canonical}). We denote in particular the \emph{canonical sheaf} by:
	\[
	\omega_X\coloneqq\Omega^{[\dim X]}_X,
	\]
	it is a divisorial sheaf whose associated linear equivalence class of Weil divisors is the \emph{canonical divisor} $K_X$ of~$X$. Following the definition given in~\cite[p.282]{Reid_canonical} for the divisorial sheaf associated to a Weil divisor, we have $\omega_X = \cO_X(K_X)$.
	With a slight, but customary abuse of notation, we denote:
	\[
	\omega_X^{[i]} = \cO_X(i K_X) = \left(\omega_X^{\otimes i}\right)^{\ast\ast}, \quad\forall i\geq 0.
	\]

	Recall that a normal variety $X$ has \emph{rational singularities} if for any resolution of singularities $f\colon Z\to X$, one has $R^{>0}f_\ast \cO_Z = 0$.
	By the extension theorem of Kebekus--Schnell~\cite[Corollary~1.8]{KS}, if $X$ is a normal variety with rational singularities, the pullback over $X_\reg$ of any reflexive form
	extends regularly on any resolution of singularities.
	We also recall that by the Elkik--Flenner theorem, 
	rational Gorenstein singularities are exactly canonical singularities of index one (see~\cite[\S3(C) \& p.363]{ReidYPG}, \cite[Claim 2.3.1]{Kollar} and \cite[Lemma 5.12]{KM}).
	
	Using these notions, we reformulate Beauville's definition~\cite[Definition~1.1]{Beauville_symplectic} of a projective \emph{symplectic variety} as: a normal projective variety~$X$ with rational singularities, whose regular part admits a symplectic holomorphic form. The complex dimension of $X$ is thus necessarily even.
	
	A \emph{finite cover map} is a finite and surjective morphism $f\colon \widetilde X\to X$ between normal projective varieties of the same dimension. It is called \emph{quasi-\'etale} if it is \'etale in codimension one ~\cite[Definition~3.2\&3.3]{GKP_etale}: this means that there exists a closed subset $Z\subset \widetilde X$ of codimension at least two such that the restriction $f\colon\widetilde X\setminus Z\to X$ is \'etale.
	The \emph{augmented irregularity}~\cite[Definition~3.1]{GKP_singular} of a normal projective variety~$X$ is the supremum $\tilde q(X)\in \IN\cup\{\infty\}$ of the irregularities~$q(\widetilde X)\coloneqq h^1(\widetilde X, \cO_{\widetilde X})$ for all finite quasi-\'etale covers $\widetilde X\to X$. 
	
	A \emph{Calabi--Yau variety} (CY) is a normal projective variety $X$ such that $\omega_X\cong \cO_X$,  with rational singularities and such that for any finite quasi-\'etale cover  $f\colon \widetilde X\to X$, one has  $H^0(\widetilde X, \Omega_{\widetilde X}^{[q]})= 0$ for all $0<q<\dim X$ (see~\cite[Definition~8.16.1]{GKP_singular} or~\cite[Definition~1.3]{GGK}). By Hodge symmetry, $\tilde q(X) = 0$ and $X$ is simply connected if $X$ has even dimension~\cite[Corollary 13.3]{GGK}. Nonsingular CY varieties correspond to the CY manifolds as defined in~\cite{Beauville_c1nul}.
	A weaker variant co-exists in the literature, without the condition on all \'etale covers, assuming only that $X$ is normal with rational singularities, $\omega_X\cong \cO_X$ and  $H^0( X, \Omega_{X}^{[q]})= 0$ for all $0<q<\dim X$: to avoid confusion, we call here these latter  \emph{weak CY varieties}.

	An \emph{irreducible symplectic variety} (ISV) is a projective symplectic variety $X$ admitting a  symplectic form $\sigma_X\in H^0(X,\Omega^{[2]}_X)$ such that  for any finite quasi-\'etale cover $f\colon \widetilde X \to X$, the reflexive pullback $f^{[\ast]}(\sigma_X)$, as defined in~\cite[\S II.4]{GKKP}, generates the exterior algebra of global sections of~$\Omega^{[\ast]}_{\widetilde X}$ (see~\cite[Definition~8.16.2]{GKP_singular} and~\cite[Definition~1.1]{BGL}). In particular, we have  $H^0(X,\Omega^{[2]}_X) = \IC \sigma_X$. Clearly $X$ is even-dimensional, $\tilde q(X) = 0$ and $X$~is simply connected by~\cite[Corollary~13.3]{GGK}, but $X_\reg$ is not always simply connected (see \cite[Remark 1.11]{PR}). Nonsingular ISVs coincide with the \emph{irreducible holomorphic symplectic} (IHS) manifolds introduced by Beauville~\cite{Beauville_c1nul}.
	
	A \emph{primitive symplectic variety} (PSV), as considered in particular by  Bakker--Lehn~\cite{BL_moduli} is a projective symplectic variety $X$ such that $q(X) =  0$
	and $H^0(X, \Omega^{[2]}_{X})$ is one-dimensional. ISV are PSV but the augmented irregularity of a PSV may be nonzero: take for instance the quotient of an abelian surface by~$\pm 1$, it is a PSV since its minimal resolution is a Kummer surface but by construction its augmented irregularity is positive.
	We refer to \cite{BL_global, Perego_examples} for a more detailed discussion on these definitions and their possible variants. Note that the names ISV and PSV are interchanged in Schwald~\cite[Definition~1]{Schwald}.
	
	Checking directly the vanishing of the augmented irregularity on a given variety is a difficult problem in general. We give a sufficient criterion for a variety to be CY or ISV, following~\cite{BCGPSV_Prym,PR}. 
	
	\begin{prop}
		Let $X$ be a normal projective variety with rational singularities and such that $\omega_X\cong \cO_X$, and assume that $\pi_1(X_{\reg})=\{1\}$.
		\begin{enumerate}[label=(\roman*)]
			\item If there exists a dominant rational map $\varphi \colon Z\dashrightarrow X$, where $Z$ is a CY variety, then $X$ is also a CY variety.
			
			\item If $X$ is a symplectic variety and if there exists a dominant rational map $\varphi \colon Z\dashrightarrow X$, where $Z$ is an ISV, then $X$ is also an ISV.
		\end{enumerate}
	\end{prop}
	
	\begin{proof}
		The argument for ISV is given in~\cite[Proposition 3.15]{BCGPSV_Prym}. The CY case is similar, but we write it for completeness. By the Zariski--Nagata purity theorem~\cite{Nagata, Zariski} (see also~\cite[Theorem~2.4]{Zong}), every quasi-\'etale cover is \'etale over the nonsingular locus. Since $X_\reg$ is simply connected, every quasi-\'etale cover of $X$ is an isomorphism by the Zariski Main Theorem, so we only need to show that $H^0( X, \Omega_{ X}^{[q]})= 0$ for all $0<q<\dim X$. Consider the following diagram:
		\[
		\xymatrix{Z'\ar[d] \ar[r] & X'\ar[d]  \\
			Z\ar@{-->}[r] & X }    
		\]
		where $X'$ is a resolution of singularities of $X$ and $Z'$ is a nonsingular resolution of indeterminacies of the rational dominant map $Z\dashrightarrow X'$. For all ${0<q<\dim X}$, by~\cite[Proposition 5.8]{Kebekus} applied to the dominant morphism $\gamma\colon Z'\rightarrow X$, the reflexive pullback morphism
		$\gamma^{[\ast]}\colon H^0(X,\Omega_X^{[q]})\to H^0(Z',\Omega_{Z'}^{q})$
		is injective. By Hodge symmetry~\cite[Proposition~6.9]{GKP_singular}, we have $h^0(Z',\Omega_{Z'}^{q})=h^q(Z', \cO_Z')$. Since the singularities of $Z$ are rational and since $Z'\to Z$ is birational, we have $h^q(Z',\cO_{Z'})= h^q(Z,\cO_Z)$. By Hodge symmetry again, since $Z$ has canonical singularities we have $h^q(Z, \cO_Z)=h^0(Z,\Omega_Z^{[q]})$, which is zero since $Z$ is CY. We  get $h^0(X, \Omega_X^{[q]})= 0$ for all $0<q<\dim X$.
	\end{proof}

	\section{K-torsion varieties}
	\label{s:torsion_K_trivial}
	
	Let $X$ be a normal projective variety. If $X$ has only canonical singularities and if its canonical divisor is numerically trivial, by Kawamata~\cite[Theorem~8.2]{Kawamata} we get that $K_X$ is torsion and that  $\kappa(X) = 0$ (see also~\cite[Chapter~V, Corollary 4.9]{Nakayama} for a more general statement).
	We assume now that
	the singularities of~$X$ are klt, see~\cite[Definition~2.34]{KM} and~\cite[p.43]{Kollar}. If $K_X$ is numerically effective, the abundance conjecture predicts that some multiple $dK_X$ of $K_X$ is base-point-free.
	If we further assume that $\kappa(X) = 0$, the morphism induced by the linear system $|dK_X|$ maps to a point, so $K_X$ is torsion. This motivates the following definition:

	\begin{defi}\label{def:K_torsion}
		A \emph{$K$-torsion variety} is a normal projective variety $X$ with klt\, singularities and torsion (but nontrivial) canonical divisor $K_X$. The \emph{index} of $X$ is the smallest integer~$d\geq 2$ such that $d K_X=0$.
	\end{defi}
	
	Under this assumption, the divisorial sheaf $\omega_X$ satisfies $\omega_X^{[d]} \cong \cO_X$, it is thus a reflexive $d$-th root of the structural sheaf. 
	
	Following~\cite[Definition~5.19]{KM}, if $X$ is a $K$-torsion variety of index $d$, the quasi-\'etale cyclic cover associated to $\omega_X$, with group $\mu_d$ (the order $d$ cyclic group),~is:
	\[
	Y\coloneqq \Spec \left(\oplus_{i=0}^{d-1}\omega_X^{[-i]}\right) \overset{p}{\longrightarrow} X.
	\]
	Since $\omega_X$ is locally free on $X_\reg$, the restriction  $p^{-1}(X_\reg)\to X_\reg$ 
	is \'etale, with $p^{-1}(X_\reg)\subset Y_\reg$. Following Kawamata~\cite[Section~8]{Kawamata}, we call~$Y$ the \emph{canonical cover} of~$X$. Recall that $Y$ is projective, irreducible, normal~(this is a consequence of~\cite[Proposition~2]{Reid_canonical}), it has canonical singularities~\cite[Corollary~5.21]{KM} and $K_Y = 0$. More generally, if $X$ is klt\, and $p\colon Y\to X$ is any quasi-\'etale cover with $Y$ normal, then $Y$ is also klt\, by~\cite[Proposition~5.20]{KM}.

	\begin{prop}\label{prop:class_group}
		Let $p\colon Y \to X$ be a quasi-\'etale cyclic cover of degree $d\geq 2$ between projective normal varieties.
		We have the following exact sequence of class groups:
		\begin{equation}\label{eq:exact_seq}
			\{0\}\longrightarrow \mu_d \longrightarrow \Cl(X) \longrightarrow \Cl(Y)^{\mu_d} \longrightarrow \{0\}.
		\end{equation}
		Moreover:
		\begin{enumerate}[label=(\roman*)]
			\item
			\label{prop:class_group_i}  If $ q(Y) = 0$, then the group $\Cl(X)$ is finitely generated.
			\item\label{prop:class_group_ii} If $K_Y = 0$, then $r K_X =0$ for some divisor $r$ of $d$.
			
		\end{enumerate}
	\end{prop}
	
	\begin{proof} 
		This statement extends~\cite[Proposition~2.8]{OS1} to the singular setup, using similar arguments. 
		By the Zariski--Nagata purity theorem, since $p$ is quasi-\'etale, it is \'etale over the nonsingular locus $X_\reg$ and $Y^\circ\coloneqq p^{-1}(X_\reg)$ is an open subset of $Y_\reg$ whose complementary in $Y$ has codimension at least two.  The restriction
		$p\colon Y^\circ \to X_\reg$ is an \'etale cyclic cover with group $\mu_d$. Following the methods and notation in~\cite[\S5.2]{Gro}, there are two  spectral sequences associated to $p$, applied here to the sheaf~$\cO_{Y^\circ}^\ast$:
		\begin{align*}
			\prescript{a}{}{E}_2^{i,j}&\coloneqq H^i(X_\reg, R^j p_\ast^{\mu_d}\cO_{Y^\circ}^\ast)\Rightarrow R^{i+j}\Gamma^{\mu_d}_{Y^\circ}(\cO_{Y^\circ}^\ast),\\
			\prescript{b}{}{E}_2^{i, j}& \coloneqq H^i(\mu_d, H^j(Y^\circ, \cO_{Y^\circ}^\ast))\Rightarrow R^{i+j}\Gamma^{\mu_d}_{Y^\circ}(\cO_{Y^\circ}^\ast).
		\end{align*}
		Since $\mu_d$ acts freely on $Y^\circ$, the functor $p_\ast^{\mu_d}$ is exact so the first spectral sequence gives $R^{i}\Gamma^{\mu_d}_{Y^\circ}(\cO_{Y^\circ}) = H^{i}(X_\reg,  p_\ast^{\mu_d}\cO_{Y^\circ})$. The five-terms exact sequence associated to the second spectral sequence writes:
		\[
		0\rightarrow	\prescript{b}{}{E}_2^{1, 0}\rightarrow H^1(X_\reg,  p_\ast^{\mu_d}\cO_{Y^\circ}^\ast) \rightarrow 	\prescript{b}{}{E}_2^{0, 1} \rightarrow 	\prescript{b}{}{E}_2^{2, 0}\rightarrow  H^2(X_\reg,  p_\ast^{\mu_d}\cO_{Y^\circ}^\ast).
		\]
		Let us compute the first terms of this exact sequence.
		\begin{itemize}
			\item Since $Y$ is normal and projective, we get:
			\[
			\prescript{b}{}{E}_2^{1, 0} = H^1(\mu_d, H^0(Y^\circ, \cO_{Y^\circ}^\ast)) =  H^1(\mu_d, H^0(Y, \cO_{Y}^\ast)) = H^1(\mu_d, \IC^\ast), 
			\]
			where $\IC^\ast$ is considered as a trivial $\mu_d$-module. An easy computation of the group cohomology of $\mu_d$ gives $\prescript{b}{}{E}_2^{1, 0}  \cong \mu_d$.
			
			\item We have $p_\ast^{\mu_d}\cO_{Y^\circ}^\ast = \cO_{X_\reg}^\ast$  and since $X$ is normal we get:
			\[
			H^1(X_\reg,  p_\ast^{\mu_d}\cO_{Y^\circ}^\ast) = H^1(X_\reg, \cO_{X_\reg}^\ast) = \Pic(X_\reg) = \Cl(X).
			\]
			
			\item Since $Y$ is normal, we get:
			\[
			\prescript{b}{}{E}_2^{0, 1} = H^0(\mu_d, H^1(Y^\circ, \cO_{Y^\circ}^\ast)) = \Pic(Y^\circ)^{\mu_d} = \Cl(Y)^{\mu_d}.
			\]
			
			\item Similarly,
			\[
			\prescript{b}{}{E}_2^{2, 0} =  H^2(\mu_d, H^0(Y^\circ, \cO_{Y^\circ}^\ast)) = H^2(\mu_d, \IC^\ast) = \{0\}.
			\]
			We thus obtain the exact sequence:
			\[
			\{0\}\longrightarrow \mu_d \longrightarrow \Cl(X) \overset{p^\ast}{\longrightarrow} \Cl(Y)^{\mu_d} \longrightarrow \{0\}.
			\]
		\end{itemize}
		\noindent \textit{Proof of assertion~\ref{prop:class_group_i}.} The subgroup $\Cl_a(Y)$ of $\Cl(Y)$ parametrizing linear equivalence classes of Weil divisors on $Y$ algebraically equivalent to zero admits a structure of abelian variety defined by Lang~\cite[IV\S4]{Lang} as the Picard variety $\PP(Y)\coloneqq \Alb(Y)^\vee$, the dual of the Albanese variety of $Y$ (see for instance~\cite{BGS} for more details). The dimension of the Picard variety is thus $q(Y) = h^1(Y, \cO_Y) = 0$ by assumption, so $\Cl_a(Y) = \{0\}$. Since the N\'eron--Severi group $\NS(Y)\coloneqq \Cl(Y)/\Cl_a(Y)$ is finitely generated by~\cite[Th\'eor\`eme~3]{Kahn}, we obtain that $\Cl(Y)$ is finitely generated. Using the exact sequence~\eqref{eq:exact_seq} we get that $\Cl(X)$ is also finitely generated. 
		
		\noindent\textit{Proof of assertion~\ref{prop:class_group_ii}.}
		To prove that $dK_X = 0$, it is enough to show that ${dK_{X_\reg} = 0}$ since $X$~is normal, so after restricting to the \'etale cover ${p\colon Y^\circ\to X_\reg}$, we may assume that $Y$~and~$X$ are nonsingular. Since $p$ is \'etale, we have $p^\ast \cO_X(K_X) = \cO_Y(K_{Y}) = \cO_{Y}$ and the result follows using the exact sequence~\eqref{eq:exact_seq}. Here is also a more direct argument that will be used below. Since $p$ is finite and $X$ is nonsingular, $p$~is flat so $p_\ast\cO_{Y}$ is locally free of rank $d$. By projection formula we get:
		\[
		p_\ast \cO_{Y} = p_\ast\left(p^\ast \cO_X(K_X)\right) = \cO_X(K_X)\otimes p_\ast \cO_{Y}.
		\]
		Since $p_\ast\cO_{Y}$ is locally free of rank $d$, by taking the determinant we get:
		\[
		\det(p_\ast\cO_{Y}) = \det\left(\cO_X(K_X)\otimes p_\ast \cO_{Y}\right)\cong \cO_X(K_X)^{\otimes d}\otimes \det(p_\ast\cO_{Y}),
		\]
		so finally $\cO_X(K_X)^{\otimes d} \cong \cO_X$, hence $dK_X  = 0$ and $rK_X =0$ for some divisor~$r$ of~$d$. 
	\end{proof}
	
	\section{Enriques manifolds and log-Enriques varieties}
	
	Inspired by~\cite{BNWS, OS1, Yoshikawa}, we give the following general definition of Enriques manifolds. We intend to give a sufficiently flexible definition
	to ensure that there are enough examples to make the theory interesting and to prepare the generalization given below to the singular setup using a similar point of view.
	
	\begin{defi}\label{def:enriques}
		An \emph{Enriques manifold} is a nonsingular $K$-torsion variety $X$ such that $q(X) = 0$.
	\end{defi}

	The condition of vanishing irregularity is needed to exclude bielliptic surfaces, which are nonsingular $K$-torsion projective varieties of index $2, 3, 4$ or $6$ and positive irregularity and observe that an Enriques manifold of dimension two is precisely an Enriques surface.
	
	Let $X$ be an Enriques manifold and consider its canonical cover $Y\to X$ of degree $d\geq 2$, which is a projective manifold with trivial canonical divisor. By the Beauville-Bogomolov decomposition theorem~\cite{Beauville_c1nul}, up to an \'etale cover, the variety~$Y$ splits as a product of IHS manifolds, CY manifolds and an abelian manifold. Since there is in general not much control on this \'etale cover, we focus on special types of Enriques manifolds where the variety $Y$ itself does not split, and we are interested in the ``atomic" types.
	
	\begin{defi}
		An Enriques manifold is of \emph{symplectic type} (\resp \emph{CY type}) if it admits a cyclic \'etale cover by an IHS (\resp CY) manifold.
	\end{defi}
	
	Observe that Enriques surface of symplectic type or CY type coincide since $\SU(2) = \Sp(1)$. In higher dimension these types exclude mutually (see Lemma~\ref{lem:types} below).
	Our definition is less restrictive than that of~\cite{BNWS} since we do not require any condition on the holomorphic Euler characteristic of $X$. It is also more general than that of \cite{OS1} which is what we call here Enriques manifolds of symplectic type, since every \'etale quotient of a projective IHS manifold is cyclic (see \cite[Lemma 2.3, Proposition 2.4]{OS1}). We will prove below that for both types, it is equivalent to ask that the canonical cover is of the given type. 
	
	We could \emph{a priori} define a third \emph{abelian type} by the condition of the existence of a cyclic \'etale cover by an abelian manifold, as a special case inside the subfamily of generalized hyperelliptic varieties, called Bagnera -- de Franchis varieties (see~\cite{Demleitner}). However we show below that this does not exist.
	
	\begin{lemma}\label{lem:no_etale_cyclic_abelian}
		Let $X$ be an Enriques manifold. Then $X$ admits no cyclic \'etale cover by an abelian variety.	
	\end{lemma}
	
	\begin{proof} 
		We use an argument taken from~\cite[Remark~2.8]{Demleitner}, communicated to us by Martina Monti. Let $A = V/\Lambda$ be an abelian variety, where $\Lambda$ is a lattice in a finite dimensional complex vector space $V$, and assume that $g\in \Aut(A)$ is an automorphism acting freely such that $A/\langle g\rangle = X$. Decomposing $g(x) = \alpha x + b$ where $\alpha\in \GL(V)$ and $b\in\Lambda$, the freeness is equivalent to the property that there exist no $(x, \ell)\in V\times\Lambda$ such that $(\alpha -\id)x = \ell - b$, so $\alpha-\id$ is not invertible. Let $z$~be an eigenvector of~$\alpha$ for the eigenvalue~$1$. Then $g^\ast dz= dz$ and hence $dz$ descends to an element of $H^0(X, \Omega^1_X) \cong H^1(X, \cO_X)$, so $q(X) >0$: this is a contradiction.
	\end{proof}
	
	\begin{prop}\label{prop:chi}
		Let $p\colon Y\to X$ be a cyclic \'etale cover of degree $d\geq 2$ between nonsingular varieties. Then  $\chi(\cO_Y) = d\chi(\cO_X)$. In particular, if $Y$ is an IHS manifold, then $d$~divides~$\frac{1}{2}\dim Y+1$.
	\end{prop}
	
	\begin{proof}The statement and the argument follow~\cite[\S2]{OS1}.
		Since $p$ is a finite projective morphism, we have $\chi(\cO_Y) = \chi(p_\ast\cO_Y)$. Since $p\colon Y\to X$ is an \'etale cyclic cover of order $d$, there exists a line bundle $L$ on $X$ such that $p_\ast \cO_{Y}\cong \cO_{X}\oplus L\oplus\cdots\oplus L^{d-1}$ with $L^{\otimes d}\cong \cO_{X}$.
		We thus have $\chi(\cO_Y) = \sum_{i=0}^{d-1}\chi(L^{i})$. Since $X$ is nonsingular, by the Hirzebruch--Riemann--Roch theorem, the Euler characteristic of $L^i$ depends only on the rational numerical class of $L$, which is trivial since $L^{d}\cong \cO_X$. We thus get $\chi(\cO_Y)= d \chi(\cO_X)$.	
		
		By Hodge symmetry, we have:
		\[
		\chi(\cO_Y) = \sum_{i=0}^{\dim Y} (-1)^i \dim H^i(Y, \cO_Y)
		=\sum_{i=0}^{\dim Y} (-1)^i \dim H^0(Y, \Omega_Y^{i}).
		\]
		If $Y$ is an IHS manifold, the spaces $H^0(Y, \Omega_{Y}^{i})$ are zero for odd $i$ and are generated by $\omega_Y^{i/2}$ for even $i$, so $\chi(\cO_Y) = \frac{1}{2}\dim Y$+1. Then the index $d$ divides $\frac{1}{2}\dim Y+1$.
	\end{proof}

	\begin{lemma}\label{lem:types}
		Enriques manifolds of symplectic or CY type are even-dimensional and the two types exclude each other if the dimension is strictly greater than $2$.
	\end{lemma}
	
	\begin{proof}
		An Enriques manifold $X$ of IHS type is clearly even-dimensional. If $X$ is of CY type, let $Y\to X$ be a cyclic \'etale cover of degree $d\geq 2$, where $Y$ is a CY manifold. If $Y$ is odd-dimensional, then its Euler characteristic is $\chi(Y, \cO_Y) = 0$, so $\chi(X,\cO_X) = 0$  by Proposition~\ref{prop:chi}. But since the quotient is \'etale, $H^0(X,\Omega_X^p) = H^0(Y, \Omega_Y^p)^\inv = 0$ for $0<p<n$, and $H^0(X, \omega_X) = 0$ since $K_X$ is torsion, giving $\chi(X, \cO_X) = 1$: this is a contradiction so $Y$ is even-dimensional. Since even dimensional CY manifolds and IHS manifolds are simply connected, in each type the cyclic \'etale cover of the definition is the universal cover, hence both types exclude each other. 
	\end{proof}
	
	\begin{cor}\label{cor:CY_canonical}
		An Enriques manifold is of CY type if and only if its canonical cover is a CY manifold, and it admits a unique cyclic \'etale CY cover. Its index is always two.
	\end{cor}

	\begin{proof}
		Let $X$ be an Enriques manifold of CY type and let  $Y\to X$ be a cyclic \'etale cover of degree $d\geq 2$, where $Y$ is a CY manifold. Since $Y$ is even-dimensional by Lemma~\ref{lem:types}, we have $\chi(Y, \cO_Y) = 2$ so the relation $d\chi(X, \cO_X) = 2$ gives $d=2$ and $\pi_1(X) \cong \IZ/2\IZ$. It follows that $2K_X =0$ (see the proof of Proposition~\ref{prop:class_group}) and by assumption $K_X$ is non trivial. The canonical cover $Z\to X$ is an \'etale double cover and since $Y$ is simply connected, the  factorization $Y\to Z$ is an isomorphism by Zariski Main Theorem.
	\end{proof}

	To extend the notion of Enriques manifold to the singular setup, in a similar spirit as above, inspired by~\cite[Definition~1.1]{Zhang} we give the following definition:

	\begin{defi}\label{def:log_enriques}
		A \emph{logarithmic Enriques variety} (\emph{log-Enriques} for short) is a $K$-torsion variety $X$ such that $ q(X) = 0$.
	\end{defi}
	
	Recall that by assumption (see Definition~\ref{def:K_torsion}) log-Enriques varieties have klt singularities.
	The term \emph{log} comes from the fact that if $\widehat X$ is a resolution of singularities of $X$ with exceptional divisor $D$, then the pair $(\widehat X, D)$ is log-terminal in the sense of \cite[Definition 2.34]{KM}. Observe that  log-Enriques varieties of dimension 2 coincide with log-Enriques surfaces as defined by Zhang~\cite{Zhang, Zhang2} except for the case of trivial canonical divisor, which is excluded in our definition.

	Let $X$ be a log-Enriques variety and consider its canonical cover $Y\to X$. By the decomposition theorem of H\"oring--Peternell~\cite{HoringPeternell}, up to a quasi-\'etale cover, the variety $Y$ splits as a product of ISVs, (singular) CY varieties and an abelian manifold (recall that abelian varieties are never singular). As in the nonsingular setup, we focus on the ``atomic" types:
	
	\begin{defi}
		A log-Enriques variety is of \emph{symplectic type} (\resp \emph{CY type}, \emph{abelian type}) if it admits a cyclic quasi-\'etale cover by a symplectic variety (\resp a CY variety, an abelian manifold).
	\end{defi}
	
	In the symplectic setup, if the cyclic quasi--\'etale cover $Y$ is \resp an IHS manifold, an ISV  or a  PSV, we will make the terminology more precise by saying that $X$ is a log-Enriques variety \resp of  \emph{IHS type}, \emph{ISV type} or \emph{PSV type}. 
	
	Log-Enriques surfaces of ISV type or CY type coincide since $\SU(2) = \Sp(1)$. In higher dimension these types exclude mutually:
	
	\begin{lemma}
		Log-Enriques varieties of ISV, CY or abelian types of dimension strictly  greater than two exclude each other.
	\end{lemma}
	
	\begin{proof}
		The argument has been communicated to us by Nikolaos Tsakanikas. 
		By~\cite[Lemma~2.19]{GGK} and~\cite[Proposition 2.10]{Nakayama-Zhang}, the augmented irregularity of a log-Enriques variety $X$ is finite (precisely $\widetilde q(X)\leq \dim X$) and it is invariant under quasi-\'etale covers. If $X$ is of ISV or CY type, then $\widetilde q (X)  = 0$ so log-Enriques varieties of ISV type or of CY type cannot be simultaneously of abelian type. Moreover, by \cite[Proposition~F]{GGK} the restricted holonomy of a CY variety is a special unitary group whereas the one of an ISV is a symplectic group, and by~\cite[Lemma~4.7]{GGK} the restricted holonomy is invariant under quasi-\'etale covers, so the ISV types and the CY types exclude each other in dimension greater than two.
	\end{proof}
	
	This result is similar to the nonsingular setup, but log-Enriques varieties of PSV type may well be of other types too.
	In a different spirit, we give in \S\ref{ss:two_types} an example of a log-Enriques variety of IHS type that is a quasi-\'etale quotient of a weak Calabi--Yau variety.

	\section{Log-Enriques varieties of symplectic type}
	\label{s: symplectic type}

	Let $Y$ be a PSV and $\sigma_Y$ be a reflexive symplectic form on $Y$. Since the space $H^0(Y,\Omega^{[2]}_Y)$ is one dimensional, for any automorphism $\varphi\in \Aut (Y)$ there exists $\lambda\in\IC^\ast$ such that $\varphi^{[\ast]}\sigma_Y  = \lambda\sigma_Y$. We call $\varphi$ \emph{symplectic} if $\lambda = 1$, and \emph{nonsymplectic} otherwise. In particular, if $\varphi$ has order $d$, we call it \emph{purely nonsymplectic} if $\lambda$ is a primitive $d$-th root of unity. A finite order automorphism $\varphi$ on $Y$ is called \emph{free} if its fixed locus is empty, and \emph{\'etale} if all its nontrivial powers are free, or equivalently  if the quotient morphism $Y\to Y/\langle \varphi\rangle$ is \'etale. Clearly every free automorphism of prime order is \'etale. It is called \emph{quasi-\'etale} if the quotient morphism $Y\to Y/\langle \varphi\rangle$ is quasi-\'etale.

	\begin{prop}\label{prop:sympl_fix_point}
		Let $Y$ be a   
		PSV and $\varphi\in\Aut(Y)$ be a symplectic automorphism of order $d\geq 2$ and let $X\coloneqq Y/\langle \varphi\rangle$. Then the quotient $X$ is PSV, hence  $K_X = 0$. If moreover $Y$ is ISV, then $X$ is ISV and  $\chi(\cO_Y) = \chi(\cO_X)$. In particular, a symplectic automorphism of finite order of an IHS manifold is never \'etale.
	\end{prop}
	
	\begin{proof}  The argument is inspired by the proof of~\cite[Lemma~2.3]{OS1}.  The variety $X$ is normal and projective, with klt\, singularities by \cite[Proposition~5.20]{KM}. We keep notation as in the proof of Proposition~\ref{prop:class_group}. The fact that $X$ is PSV (\resp ISV)  has been proven in \cite[Proposition 2.4]{Beauville_symplectic} (\resp in \cite[Proposition 3.17]{BGMM}). For later use, we show here that $K_X=0$. Since $\varphi$ is symplectic, the symplectic form on $Y_\reg$ goes down to a symplectic form on $X_\reg$, so $K_X = 0$.
		
		Since $\varphi$ is symplectic, the regular parts of the fixed loci of its nontrivial powers are symplectic submanifolds by \cite[Proposition~3.10]{BCGPSV_Prym} so their codimensions are at least two, hence $\varphi$ is quasi-\'etale and the quotient map $p$ restricts to an \'etale cyclic cover $p\colon Y^\circ\to X_\reg$.  The restriction maps $H^0(Y, \Omega^{[i]}_Y)\to H^0(Y^\circ, \Omega^{i}_{Y^\circ})$ are isomorphisms, generated on the right hand side by the restrictions of $\wedge^i \sigma_Y$ to $Y^\circ$ since $Y$ is  an ISV. If $\varphi$ is symplectic, all these spaces are $\varphi$-invariant so we get by Hodge symmetry for singular spaces:
		\[
		\chi(\cO_Y) =\sum_{i=0}^{\dim Y} (-1)^i \dim H^0(Y^\circ, \Omega_{Y^\circ}^{i})^{\langle\varphi\rangle}.
		\]
		Since $p$ is \'etale on $Y^\circ$, we have $\Omega^1_{Y^\circ}\cong p^\ast \Omega_{X_\reg}^1$, so  $\Omega^i_{Y^\circ}\cong p^\ast \Omega_{X_\reg}^i$ for all $i$ and we obtain:
		\begin{align*}
			\chi(\cO_Y) &=\sum_{i=0}^{\dim Y} (-1)^i \dim H^0(Y^\circ, p^\ast \Omega_{X_\reg}^i)^{\langle\varphi\rangle}
			=\sum_{i=0}^{\dim Y} (-1)^i \dim H^0(X_\reg, \Omega_{X_\reg}^i),
		\end{align*}
		where the last equality is a standard fact on sections of $\langle\varphi\rangle$-linearized sheaves obtained by pullback from a quotient. Using Hodge symmetry as above, this time for $X$, we get $\chi(\cO_Y) = \chi(\cO_X)$. 
		
		In particular, if $Y$ is an IHS  manifold and if $\varphi$ is symplectic and \'etale of order $d\geq 2$, then using Proposition~\ref{prop:chi}
		we get $\frac{1}{2}\dim Y + 1 =\chi(\cO_Y) = \chi(\cO_X) = d \chi(\cO_Y)$: this is a contradiction, so $\varphi$ cannot be \'etale.
	\end{proof}
	
	\begin{example}
		Let $Y$ be a projective IHS manifold that is deformation equivalent to the Hilbert scheme of two points on a K3 surface. Assume that $Y$ admits a symplectic involution $\iota$. Then the quotient $Y/\langle\iota\rangle$ has trivial canonical bundle and it admits a partial resolution that is an ISV (more precisely an irreducible symplectic orbifold), called a \emph{Nikulin orbifold}~\cite[Definition~3.1]{CGKK}.
	\end{example}
	
	\begin{prop}
		\label{prop:conseq1}  Let $Y$  be a PSV (\resp ISV) of dimension $2n$ with a quasi-\'etale purely non symplectic automorphism $\varphi$ of order $d$.
		If $d$ divides $n$ then $K_X = 0$, otherwise  $X\coloneqq Y/\langle \varphi\rangle$ is a log-Enriques variety of PSV type  (\resp ISV type), and index $r= d/\gcd(d, n)$.
	\end{prop}
	
	The quasi-\'etale assumption in the statement of Proposition~\ref{prop:conseq1} is automatic whenever $\dim Y>2$, since the singular locus has codimension at least two and 
	the regular part of the fixed loci in $Y_\reg$ of all the nontrivial powers of $\varphi$ are isotropic submanifolds. This is well-known to the experts, but due to a lack of reference we state and prove this result below.
	
	\begin{lemma}\label{lem:isotropic-fixed-locus}
		Let $Y$ be a PSV of dimension at least $4$ with symplectic form $\sigma_Y$ and let $\varphi$ be a  nonsymplectic automorphism of finite order on $Y$. Then any irreducible component  of the fixed locus of $\varphi$ in $Y_\reg$ is isotropic for $\sigma_Y$.
	\end{lemma}
	
	\begin{proof} 
		Let $\lambda\in \IC$ be such that $\varphi^{[\ast]}\sigma_Y = \lambda \sigma_Y$; by assumption $\lambda\neq 1$. Let $Z\coloneqq Y_\reg$ and $F\subset Z$ be an irreducible component of the fixed locus of $\varphi$ in $Z$.
		Consider the splitting of the restriction of the tangent bundle: $(T_Z)_{|F}=T_F\oplus N_{F|Z}$. Since $\varphi$ is of finite order, $T_F$ coincides with the $(+1)$-eigenbundle of the action of $\varphi$ on $(T_Z)_{|F}$.  For any $x\in F$ and for any $u,v\in T_{F,x}$ we thus get:
		\[\sigma_X(u,v)=\sigma_X(d\varphi_x(u),d\varphi_x(u))=(\varphi^\ast\sigma_Y)
		(u, v)=\lambda \sigma_X(u,v),
		\]
		hence, $\sigma_X(u,v)=0$ since $\lambda\neq 1$.
	\end{proof}

	\begin{proof}[Proof of Proposition~\ref{prop:conseq1}]
		Denote the quasi-\'etale quotient by $p\colon Y\to Y/\langle \varphi\rangle\coloneqq X$.
		The variety~$X$ is normal and projective, with klt\, singularities by \cite[Proposition~5.20]{KM}.
		By Proposition~\ref{prop:class_group}\ref{prop:class_group_ii}, the torsion index $r$ of the Weil divisor $K_X$ divides $d$. Using similar arguments as in the proof of Proposition~\ref{prop:sympl_fix_point}, we have:
		\begin{align*}
			H^0(X, \omega_X) &= H^0(X_\reg, \omega_{X_\reg}) = H^0(Y^\circ, p^\ast\omega_{X_\reg})^{\langle \varphi\rangle}\\ 
			&= H^0(Y^\circ, \omega_{Y^\circ})^{\langle \varphi\rangle} =H^0(Y, \omega_{Y})^{\langle \varphi\rangle}. 
		\end{align*}
		The space $H^0(Y, \omega_{Y})$ is generated by $\wedge^{n}\sigma_Y$. Since $\varphi$ is purely nonsymplectic of order $d$, we have $\varphi^\ast\left(\wedge^{n}\sigma_Y\right) = \xi^n \left(\wedge^{n}\sigma_Y\right)$, where $\xi$ is a primitive $d$-th root of unity. 
		If $d$ divides $n$, the space $H^0(X, \omega_X)$ contains a regular global section that does not vanish, so $K_X = 0$. Otherwise this global section is not $\varphi$-invariant so ${H^0(X, \omega_X) = 0}$ and $K_X$ is not trivial.  Since $rK_X= 0$, a similar computation as above gives:
		\[
		\IC = H^0(X, \omega_X^{\otimes r})\cong H^0(Y, \omega_Y^{\otimes r})^{\langle \varphi\rangle}.
		\] 
		The space $H^0(Y, \omega_Y^{\otimes r})$ is generated by the nonvanishing section $\left(\wedge^{n}\sigma_Y\right)^{\otimes r}$, that is multiplied by $\xi^{rn}$ under the action of $\varphi^\ast$, so this global section is $\varphi$-invariant if and only if $d$ divides $rn$. A similar argument shows that the irregularity of $X$ vanishes.
        
        It follows that the index $r$ is the smallest integer satisfying both conditions $r$ divides $d$ and $d$ divides $rn$. It is easy to see that  $r=d/\gcd(d, n)$ is the smallest solution (this argument originates in~\cite[proof of Proposition~2.8]{OS1}). 
	\end{proof}

	\begin{rem}\label{rem:index} We add two observations to Proposition~\ref{prop:conseq1}:
		\begin{enumerate}
			\item We have $r=d$ if and only if $d$ and $n$ are coprime.  In particular, it is not possible to construct an index six (log)-Enriques variety as a quotient of a  four-dimensional IHS manifold by an order $6$ automorphism, see~\cite[Remark~4.1]{BNWS} and~\cite[Remark~6.6]{OS1}, but see~\S\ref{ss:index6} and~\S\ref{ss:index12} for examples of log-Enriques varieties of index $6$ or $12$ using higher-dimensional quotients. 
			
			\item If $Y$ is an IHS manifold and  $\varphi$ is \'etale, by Proposition~\ref{prop:chi} futhermore $d$ divides $n+1$, so we have always $r=d$.
		\end{enumerate}
	\end{rem}
	
	\begin{cor}\label{cor:etale_is_canonical}
		An Enriques manifold is of symplectic type if and only if its canonical cover is an IHS manifold, and it admits a unique cyclic \'etale IHS cover.
	\end{cor}
	
	\begin{proof}
		Let $X$ be an Enriques manifold of symplectic type and let $Y\to X$ be a cyclic \'etale cover of degree $d\geq 2$, where $Y$ is an IHS manifold. Following~\cite{OS1}, the cyclic group is generated by a purely nonsymplectic automorphism, otherwise some of its powers would be symplectic, and using the  holomorphic Lefschetz formula we see that the quotient map would be ramified. By Remark~\ref{rem:index}, the index of the canonical divisor of $X$ is equal to $d$. The  canonical cover of $Z\to X$ has thus order $d$ and, since $Y$ is simply connected, $Y$ is isomorphic to $Z$.
	\end{proof}
	
	\begin{example} \label{ex:CGM}
		Let $S$ be a projective K3 surface with a nonsymplectic involution~$\iota$. The natural involution $\iota^{[2]}$ induced on the Hilbert square $Y:=S^{[2]}$ of $S$ is nonsymplectic. Its fixed locus is non empty, even if $\iota$ is an Enriques involution. The quotient map $S^{[2]}\to S^{[2]}/\langle\iota^{[2]}\rangle\eqqcolon X$ is quasi-\'etale and $K_X = 0$ by Proposition \ref{prop:conseq1}. Similarly as in the proof of Proposition~\ref{prop:sympl_fix_point}, we have:
		\[
		\chi(\cO_X) = \sum_{i=0}^2 \dim H^0(S^{[2]}, \Omega_{S^{[2]}}^{[2i]})^{\langle \varphi\rangle}.
		\]
		The symplectic form $\sigma_Y$ generates $H^0(S^{[2]}, \Omega_{S^{[2]}}^{[2]})$ and is not invariant, but
		its exterior product $\wedge^2\sigma_Y$ is $\varphi$-invariant and generates $H^0(S^{[2]}, \Omega_{S^{[2]}}^{4})$,
		so $\chi(\cO_X) =  2$ whereas $\chi(\cO_Y) = 3$. Camere--Garbagnati--Mongardi~\cite{CGM} prove that $X$ admits a crepant resolution by a Calabi--Yau manifold. By Proposition~\ref{prop:class_group}, we have an exact sequence:
		\[
		\{0\}\longrightarrow \mu_2 \longrightarrow \Cl(X) \longrightarrow \Pic(S^{[2]})^{\mu_2} \longrightarrow \{0\}
		\]
		and $\Pic(S^{[2]})$ is torsion-free. See \S\ref{ss:Enriques_involution} for a variant of this construction producing Enriques manifolds.
	\end{example}
	
	\begin{prop}\label{prop:conseq2} 
		Let $X$ be a  log-Enriques variety of PSV (\resp ISV) type of index $d$, and let $p\colon Y\to X$ be a cyclic quasi-\'etale cover such that $Y$ is a PSV (\resp ISV). Assume that $p$ has degree $d$. Then $X$ is the quotient of~$Y$ by a purely non symplectic automorphism of order $d$. If moreover $\Cl(Y)$ is torsion-free, then $Y$ is the canonical cover of $X$.
	\end{prop}

	\begin{proof}
		Let $\varphi$ be any generator of the cyclic group $\mu_d$ acting on $Y$, that is quasi-\'etale by assumption.  If $\varphi$ is symplectic, then by Proposition ~\ref{prop:sympl_fix_point} we have $K_X = 0$, this is a contradiction. Similarly, if $\varphi^k$ acts symplectically on 
		$Y$ for some strict divisor~$k$ of~$d$  we may factorize the quotient as follows:
		\[
		p\colon Y\longrightarrow \overline{Y}\coloneqq Y/\langle \varphi^k\rangle\overset{q}{\longrightarrow} \overline{Y}/\langle \overline\varphi\rangle = X,
		\]
		where $\overline \varphi$ is the class of $\varphi$ in $\langle \varphi\rangle/\langle\varphi^k\rangle$. Then $K_{\overline Y} = 0$ by Proposition~\ref{prop:sympl_fix_point} and a similar computation as in the proof of Proposition~\ref{prop:class_group}\ref{prop:class_group_ii} shows that  $\frac{d}{k}K_X = 0$. This is a contradiction, so $\varphi$ is purely nonsymplectic.
		
		Since $p$ is a cyclic quotient, there exists a reflexive sheaf $L$ on $X$ that is a $d$-torsion element in $\Cl(X)$ such that $Y = \Spec \oplus_{i=0}^{d-1}L^{[-i]}$. By Proposition~\ref{prop:class_group}, 
		if $\Cl(Y)$ is torsion-free, the torsion part of $\Cl(X)$ is generated by~$\omega_X$. It follows that the $\cO_X$-algebras $\oplus_{i=0}^{d-1}L^{[-i]}$ and $\oplus_{i=0}^{d-1}\omega_X^{[-i]}$ are equal, so $Y$ is the canonical cover of~$X$.
	\end{proof}
	
	\begin{rem} As remarked in the proof of Proposition \ref{prop:sympl_fix_point}, by a recent result of Bertini, Grossi, Mauri and Mazzon \cite[Proposition~3.17]{BGMM},  quotients of ISV by finite order automorphisms acting symplectically are again ISV, so in our proof, if $Y$ is an ISV then  $\overline Y$ is an ISV. In particular, since $Y$ and $\overline Y$ are then simply connected  by~\cite[Corollary~13.3]{GGK}, the quotient by $\varphi^k$ is quasi-\'etale, but certainly not \'etale. This remark holds for every quotient of an ISV by a finite order automorphism acting symplectically, so that the second assertion of Proposition \ref{prop:sympl_fix_point} extends to the singular setting: a symplectic automorphism of finite order of an ISV is never \'etale. 
	\end{rem}

	Let us recall some classical basic facts about the topology of a singular quasi-projective complex algebraic variety $X$. When working with the classical topology, $X$~is Hausdorff, connected and locally path-connected and it has the homotopy type of a finite CW-complex, so it is in particular locally simply connected. Thus it admits a universal cover (unramified) $p\colon\widetilde X \to X$ that inherits a complex analytic structure such that $p$ is a local analytic isomorphism, but in general $\widetilde X$ is not algebraic.
	Assuming that $X$ is normal, it is locally irreducible for the classical topology: it follows that $\widetilde X$ is irreducible and that the natural map $\pi_1(X_\reg)\to \pi_1(X)$ is surjective (see~\cite[\S0.7]{FultonLazarsfeld} and~\cite[\S2]{ADH} and references therein). 
	
	If the cyclic cover $p\colon Y \to X$ defining a log-Enriques variety of ISV type and index $d$ is \'etale, then clearly $\pi_1(X)\cong \IZ/d\IZ$. Otherwise, the topological fundamental group of $X$ may be smaller than expected. In particular, if the group $\mu_d$ acting on~$Y$ has a fixed point, then $X$ is simply connected (this is a special case of a result due to Kaneko--Yoshida, see~\cite[Lemma~1.2]{Fujiki}). 
	
	Enriques manifolds are defined in~\cite{OS1} as projective, nonsingular varieties, nonsimply connected, whose universal cover is an IHS manifold. In the singular setup, the behavior of the fundamental group and of the universal cover pushed us to use instead the properties of the canonical sheaf to define log-Enriques varieties. However, the following result remains true in the singular setup, but it is not a characterization any more.

	\begin{prop}\label{prop:OS}\text{}
		Let $X$ be a $K$-torsion variety, non simply connected, whose universal cover $Y$ is a PSV (resp. an ISV). If the fundamental group $\pi_1(X)$ acts purely nonsymplectically on $Y$, then $X$ is a log-Enriques variety of PSV (resp. ISV) type. 
	\end{prop}

	\begin{proof} The proof extends~\cite[Proposition~2.4]{OS1} to the singular setup, using similar arguments. 
		The fundamental group of $X$ acts biholomorphically on its universal cover~$Y$, so $\pi_1(X)$ is a subgroup of $\Aut(Y)$. The Lie algebra of $\Aut(Y)$ is isomorphic to $H^0(Y, T_Y)$ (see~\cite[Lemma 3.4]{Matsumura-Oort}). Since $Y$~is a PSV, we have $H^0(Y, T_Y)=  \{0\}$ by~\cite[Lemma 4.6]{BL_moduli},  so $\Aut(Y)$ is a discrete group. Since the quotient morphism $p\colon Y \to X = Y/\pi_1(X)$ is \'etale, the group $\pi_1(X)$ is in bijection with any fiber of~$p$. These fibers are discrete and projective, hence finite, so $\pi_1(X)$ is finite. By assumption, the group $\pi_1(X)$ acts purely nonsymplectically on $Y$, so the group morphism $\pi_1(X)\to \IC^\ast$ sending an automorphism $\varphi$ to the complex number $\xi$ such that $\varphi^{[\ast]}\omega_Y = \xi\omega_Y$ is injective. We deduce that $\pi_1(X)$ is a cyclic group. By Proposition~\ref{prop:conseq1}, since $X$ is assumed to be a $K$-torsion variety, it is a log-Enriques variety.
	\end{proof}

	\section{Examples}

	\label{s:examples_nonsingular}
	
	\subsection{Known examples of Enriques manifolds}
	
	There is a wide diversity of geometric constructions of Enriques surfaces, we refer for instance to~\cite{CossecDolgachev, Dolgachev, Mukai}. However, in higher dimension only a few geometric constructions of nonsingular Enriques varieties exist. We briefly review the constructions given in~\cite{BNWS, OS1} and give some new examples using similar methods.
	
	For any two-dimensional complex torus $A$ with origin $0\in A$, we denote by $\Sym {n+1} A$ the $n+1$-th symmetric power of $A$, by $s\colon \Sym {n+1} A \to A$ the summation map and by $\Hilb {n+1} A$ the Hilbert scheme of $n+1$ points on $A$. Consider the Hilbert--Chow morphism $\rho\colon \Hilb {n+1} A \to \Sym {n+1} A$, the \emph{generalized Kummer variety} (of dimension $2n$) of $A$ is the fiber $\Kum n A\coloneqq (s\circ \rho)^{-1}(0)$. Consider now an automorphism~$g$ of~$A$; it is well known that it is a composition $g=t_a\circ h$ of a translation $t_a$ and $h\in \Aut_{\IZ}(A)$, and $g$ restricts to $\Kum n A$ if and only if $a$~is a $n+1$-torsion point of~$A$. Let $\omega$ be a  non--degenerate holomorphic 2--form on $A$, then $g^*\omega=\lambda \omega$ for some $\lambda\in \IC^*$. Consider now the projections $pr_i: A^{n+1}\lra A$, $i=1,\ldots n+1$, (where $A^{n+1}$ denotes the $(n+1)$-times direct product $A\times\ldots\times A$ ). As shown in \cite{Beauville_c1nul}, we have that $\psi:=pr_1^*\omega+\ldots+pr_{n+1}^*\omega$ defines a symplectic form on $A^{n+1}$ that one can pull-back on the blow-up of the diagonal of $A^{n+1}$, and this comes then from a symplectic form  $\sigma$ on $A^{[n+1]}$. Hence, the induced natural automorphism $K_n(g)$ acts again by multiplication by $\lambda$ on $\sigma$ by construction. A similar remark holds for the action of a natural automorphism on $S^{[n+1]}$ for $S$ a K3 surface. For more details see \cite{Beauville_c1nul}, \cite[Subsection 3.1 and Section 4]{BNWS}.

	\subsubsection{The Hilbert scheme of points on Enriques surfaces.} Let $E$ be an Enriques surface. Then for any $n\geq 2$, the Hilbert scheme $\Hilb n E$ of $n$ points on $E$ is an Enriques manifold of CY type, of dimension $2n$ and index two~\cite[Theorem~3.1]{OS1}.

	\subsubsection{Involutions on generalized Kummer varieties of decomposable abelian surfaces.}
	Let $A=E\times F$ be the product of two elliptic curves and consider the automorphism  $\iota_A$ of $A$ defined by $\iota_A(x,y)=(-x+u, y+v)$. For any integer $n\geq 1$, the natural automorphism $\iota_A^{[n+1]}$ of $\Hilb {n+1} A$ respects the fiber $\Kum n A$ if $u$ and $v$ are $(n+1)$-torsion points of $E$ and $F$ respectively. Moreover, $\iota_A$ is an involution on $A$ if $v$ is a $2$-torsion point.  Assuming thus that $n$ is odd, $u\in E[n+1]$ and $v\in F[2]$, it is easy to check that the natural automorphism $\Kum n {\iota_A}$ is a fixed point free involution on $\Kum n A$ if $u$ is not an $\frac{n+1}{2}$-torsion point. Under these conditions, by Propositions~\ref{prop:chi}\&\ref{prop:conseq1} the quotient $\Kum n A /\langle\Kum n {\iota_A}\rangle$ is an Enriques manifold of symplectic type, of dimension $2n$ and index two.
	
	\subsubsection{Enriques involutions on K3 surfaces}\label{ss:Enriques_involution}
	
	Let $S$ be a K3 surface carrying an \emph{Enriques involution} $\iota_S$, \ie a fixed point free involution, and let $n\geq 1$ be an odd integer. Consider the natural involution $\iota_S^{[n]}$ induced on the Hilbert scheme $\Hilb n S$. If $n$ is odd, the quotient $\Hilb n S/\langle\iota_S^{[n]}\rangle$ is an Enriques manifold of dimension~$2n$, of symplectic type and index two. If $n$ is even, then by Proposition~\ref{prop:conseq1} the quotient has trivial canonical class, so it is not 
	an Enriques manifold. See Example~\ref{ex:CGM} for a description of the case $n=2$.
	
	\subsubsection{Index $2$  Enriques manifolds from moduli spaces of stable sheaves on $K3$ surfaces.}\label{sssec:moduli-stable-sheaves}
	
	Let $S$ be a very general $K3$ surface carrying an Enriques involution $\iota_S$ as above. By genericity we can assume that $\rho(S)=10$ and that the Picard lattice $\Pic(S)$ is $\iota_S$-invariant. Given a primitive Mukai vector $v=(r,l,\chi-r)\in H^*(S,\IZ)$ and a $v$-general polarization $H\in\Pic(S)$, the induced involution $\iota_S^*$  acts on the moduli space $M_{v,H}(S)$ of $H$-stable sheaves on $S$ with Mukai vector $v$. If $\chi$ is an odd integer, the quotient $M_{v,H}(S)/\langle\iota_S^*\rangle$ is an Enriques manifold of dimension $v^2+2$, of symplectic type and index two~\cite[Theorem~5.3]{OS1}.
	This is a generalization of the previous construction.
	
	\subsubsection{Complete intersections of quadrics} For any even integer $n\geq 2$, consider the projective space $\IP^{2n+1}$ with homogeneous coordinates $[x_0:\ldots:x_n:y_0:\ldots:y_n]$. For generic quadrics $Q_j(x), \widetilde Q_j(y)$ for $j=1, \ldots, n+1$, the complete intersection:
	\[
	Y \coloneqq \{[x:y]\,|\,Q_j(x) = \widetilde Q_j(y),\quad\forall j=1,\ldots, n+1\}
	\]
	is a nonsingular CY variety (we denote here $x:=(x_0,\ldots,x_n)$ and $y:=(y_0,\ldots,y_n)$). The involution~$\iota$ of $\IP^{2n+1}$ defined by $\iota(x_i) = -x_i$ and $\iota(y_i) = y_i$ for $i=0,\ldots, n$ fixes the two $n$-dimensional projective spaces defined by $\{x_i=0\,|\, i=0,\ldots,n\}$ and $\{y_i=0\,|\, i=0,\ldots,n\}$. For generic choices of the quadrics, this fixed locus does not meet $Y$, so the quotient $Y/\langle\iota\rangle$ is an Enriques manifold of CY type, of dimension $n$ and index two. Observe that if $n$ is odd, the same construction gives an involution preserving the volume form, so the quotient is a weak Calabi--Yau manifold which is not simply connected, since $\pi_1(Y/\langle \iota\rangle)=\IZ/2\IZ$.
	
	\subsubsection{Index $3$  Enriques manifolds from  generalized Kummer varieties of special decomposable abelian surfaces.}
	
	Let $A = E\times E_\omega$ where $E$ is an elliptic curve, $\omega$ is a primitive $3$rd root of unity and $E_\omega = \frac{\IC}{\IZ\oplus \omega \IZ}$. Take $u\in E$, $v\in E_\omega$ and consider the automorphism $f_A(x, y) = (x+u, \omega y+ v)$. Assume that $u$ is a $3$-torsion point, so that $f_A$ has order $3$. Assuming that $u$ and $v$ are $(n+1)$-torsion points, the natural automorphism $f_A^{[n+1]}$ of $\Hilb {n+1} A$ respects the fiber $\Kum n A$ and the necessary condition given by Proposition~\ref{prop:chi} imposes that $3$ divides $n+1$, so we write $n+1 =  3m$. If $m(2+\omega) v \neq 0$, the natural automorphism $\Kum n {f_A}$ is a fixed point free order three automorphism on $\Kum n A$ and by Proposition~\ref{prop:conseq1}, under these assumptions, the quotient $\Kum n A /\langle\Kum n {f_A}\rangle$ is an Enriques manifold of symplectic type, of dimension $2n$ and index three.
	
	A variant of this construction, starting from a quotient $A=\frac{E\times E_\omega}{\IZ/3\IZ}$, produces for an appropriate choice of the automorphism a different family of Enriques manifolds of symplectic type as quotients of generalized Kummer manifolds, see~\cite[\S 4.3.7 - Type 6]{BNWT}. 
	
	\subsubsection{Index $4$  Enriques manifolds from  generalized Kummer manifolds of special decomposable abelian surfaces.}
	
	Similarly as above, take $A = E\times E_\ii$ where $E_\ii = \frac{\IC}{\IZ\oplus \ii \IZ}$. Take $u\in E$, $v\in E_\ii$ and consider the automorphism $f_A(x, y) = (x+u, \ii y + v)$. Assume that $u$ is a $4$-torsion point, so that $f_A$ has order $4$. Assuming that $u$ and $v$ are $(n+1)$-torsion points, the natural automorphism $f_A^{[n+1]}$ of $\Hilb {n+1} A$ respects the fiber $\Kum n A$ and the necessary condition given by Proposition~\ref{prop:chi} imposes that $4$~divides~$n+1$, so we write $n+1 =  4m$. If $2mu\neq 0$ or if ${2m(1+\ii)v\neq 0}$, the natural automorphism $\Kum n {f_A}$ is a fixed point free order four automorphism on $\Kum n A$ and  by Proposition~\ref{prop:conseq1}, under these assumptions,  the quotient $\Kum n A /\langle\Kum n {f_A}\rangle$ is an Enriques manifold of symplectic type, of dimension $2n$ and index four.
	
	\subsubsection{Enriques manifolds from $OG_6$ and $OG_{10}$ manifolds} A priori one could produce examples of Enriques manifolds by considering \'etale quotients of $OG_6$ and $OG_{10}$ manifolds and their deformations. The possible indices are $2, 4$, \resp $2,3,6$. In a recent paper \cite{BGGG_OG10} the authors show that one can not obtain Enriques manifolds as quotients of $OG_{10}$ manifolds and their deformations. The case of quotients of $OG_6$ manifolds remains on the contrary open.

	\subsection{Log-Enriques varieties of IHS type}
	
	By Proposition~\ref{prop:conseq1}, any quotient of a $2n$-dimensional IHS manifold, with $n\geq 2$,  by a purely nonsymplectic automorphism is a log-Enriques variety as long as the order of the automorphism does not divide $n$. Here, we discuss some features of examples which can be constructed from some deformation families of IHS manifolds.

	\subsubsection{Log-Enriques varieties of prime index as quotients of IHS manifolds}\label{sssec:nonsymplectic quotients}
	Nonsymplectic automorphisms of prime order acting on IHS manifolds have now been classified for all known deformation families (see \cite{BCS, BCMS, BNWT, CC, CCC, BC} for $K3^{[n]}$ type, \cite{MTW} for generalized Kummer manifolds, \cite{GrossiOG6} for $OG_6$ manifolds and \cite{BG_OG10, BC} for $OG_{10}$ manifolds).
	Let $Y$ be an IHS manifold of dimension $2n\geq 4$ and let $\varphi\in\Aut(Y)$ be nonsymplectic of prime order $p$; in particular, the order satisfies $2\leq p\leq 7$ when $Y$ is either of $\textrm{Kum}_n$--type or of $OG_6$--type, \resp $2\leq p\leq 23$ when $Y$ is either of $K3^{[n]}$--type or of $OG_{10}$--type. Proposition \ref{prop:conseq1} implies that the quotient $\pi\colon Y\to Y/\langle\varphi\rangle\eqqcolon X$ is log-Enriques of index $p$ when $p$ does not divide $n$, and in this case the second Betti number $b_2(X)$ coincides with the rank of the invariant sublattice $T_\varphi\subset H^2(Y,\IZ)$. As a consequence, this construction yields $2n$-dimensional log-Enriques varieties of index $p$ such that $(p,n)=1$ and with second Betti number $1\leq b_2(X)\leq b_2(Y)-2$ and  $b_2(X)= b_2(Y) \mod p-1$.

	As explained in~\cite{BCS} and~\cite[Lemma~3.4(2)]{CGM}, given a fixed component $Z\subset Y^\varphi$ of dimension~$s$, the action of $\varphi$ linearizes near a point $q\in Z$ as follows. If $p=2$ then $s=n$ since $Z$ is Lagrangian by~\cite[Lemma~1]{Beauville_invol} and the action diagonalizes as:
	\[
	\rm{diag}\left(\underset{s \text{ times}}{\underbrace{1, \ldots, 1}}, \underset{s\text{ times}}{\underbrace{-1, \ldots, -1}}\right).
	\]
	If $p\geq 3$ then by \cite[Lemma 3.4]{CGM} the diagonalization may be written as:
	\[
	{\rm{diag}}\left(\underset{s \text{ times}}{\underbrace{{\underbracket{1, \xi_p},\ldots, \underbracket{1, \xi_p}}}}, \underset{t \text{ blocks}}{\underbrace{{\underbracket{\xi_p^{a_1}, \xi_p^{p+1-a_1}},\ldots, \underbracket{\xi_p^{a_t}, \xi_p^{p+1-a_t}}}}},\underset{2n-2s-2t}{\underbrace{\xi_p^{\frac{p+1}{2}}, \ldots, \xi_p^{\frac{p+1}{2}}}} \right),
	\]
	where $\xi_p$ is the primitive $p$-th root of unity such that $\varphi$ acts by multiplication by~$\xi_p$ on the symplectic $2$-form, with $s\leq n$ by Lemma~\ref{lem:isotropic-fixed-locus} since $Z$~is isotropic, and $1<a_j<p$ for  $1\leq j\leq t$.
    
    Recall that when the linearization of the action of  $\varphi$ at the fixed point $q$ is ${\rm diag}(\xi_p^{\alpha_1},\ldots,\xi_p^{\alpha_n})$, with $a_i\in \IZ$, then its \emph{age} is defined as
    \[
    \age(\varphi, q):=\frac{1}{p}\sum_{i=1}^n \overline{\alpha_i},
    \]
    where $^-$ is the smallest residue modulo $p$.
    Since the group generated by $\varphi$ acts without quasi-reflection, by the Reid--Sheperd-Barron--Tai criterion for quotient singularities (see~\cite[Theorem 3.1]{Reid_canonical} and \cite[\S 4.11]{ReidYPG}), the type of the quotient singularity at $\pi(q)$ is determined by the age as follows:
	\[
	\begin{cases}
		\text{canonical} & \text{if } \age(\varphi^k,  q)\geq 1 \quad\forall k=1, \ldots, p-1,\\
		\text{terminal} & \text{if } \age(\varphi^k,  q)>1\quad\forall k=1, \ldots, p-1,\\
		\text{strictly klt} & \text{if }\exists k,\quad 0<\age(\varphi^k,  q)<1.	\end{cases}
	\]
	We compute here that $\age(\varphi, q) = \frac{n}{2}$ if $p=2$. If $p>2$, then $\age(\varphi^k,q)\geq \frac{2n-s}{p}\geq\frac{n}{p}$ for any $k=1,\ldots,p-1$. Hence the log-Enriques variety $X$ has terminal singularities along the image of $Z$ if  $p$ does not divide $n$ (this trivially implies $n\neq p$) and one of the following conditions holds:
\begin{enumerate}
	\item  $n> p$,
	\item $n< p$ and  $s<2n-p$.
\end{enumerate}

    Recall that a normal projective variety with numerically trivial canonical divisor is uniruled if and only if it does not have canonical singularities,  see e.g. \cite[Lemma~2.1]{DRTX} where this property is investigated for several log-Enriques varieties. We give in~\S\ref{ss:two_types} an example of a log-Enriques variety of IHS type and index three which has strictly klt singularities, showing in this way that it is uniruled. Observe that by an easy computation, canonical non terminal singularities are not possible on these quotients.
	
	\begin{rem} By \cite[Lemma 2.3]{Zhang}
		the maximum prime index for a log-Enriques surface is $19$. In higher dimension the prime index may be bigger. In fact in \cite{BCMS} it is shown that there exists an IHS manifold of $K3^{[2]}$-type admitting an automorphism of order $23$ acting purely nonsymplectically. By using the holomorphic Lefschetz formula one easily shows that the fixed locus is not empty and by Proposition~\ref{prop:conseq1} we get a log-Enriques variety of IHS type and index $23$. 
	\end{rem}
	\subsubsection{Log-Enriques varieties of IHS type from weak CY varieties}\label{ss:two_types}
	Let $C\subset\mathbb{P}^5$ be a smooth cubic fourfold defined by an equation of the form
	\[x_0^2L(x_1,\ldots,x_4)+G(x_1,\ldots,x_4)+x_5^3=0,\]
	where $L\in\IC[x_1,\ldots, x_4]_1$ and $G\in\IC[x_1,\ldots, x_4]_3$ are \resp a linear and a cubic form. Then there exist two commuting automorphisms $\iota,\sigma\in\Aut(C)$, \resp of order two and three, given by $\iota([x_0:\ldots :x_5])=[-x_0:x_1:\ldots :x_5]$ and $\sigma([x_0:\ldots :x_5])=[x_0:\ldots :x_4 :\zeta_3x_5]$, where $\zeta_3$ is a nontrivial third root of unity.
	
	The induced automorphisms $\iota,\sigma\in\Aut(F(C))$ act nonsymplectically on the Fano variety of lines $F(C)$ (see \cite[\S 7]{C} and \cite[Example 6.4]{BCS}). 
	The quotient $Y:=F(C)/\langle\iota\rangle$ is a singular weak CY variety~\cite{CGM}
	and the order three automorphism~$\sigma$ descends to an automorphism~$\overline{\sigma}$ of the quotient~$Y$; moreover, $\overline{\sigma}$ does not preserve the volume form on $Y$. Hence, the quotient: 
	\[
	X:=Y/\langle\overline{\sigma}\rangle=F(C)/\langle\iota\circ\sigma\rangle
	\]
	is a log-Enriques variety of IHS type and index three. On the other hand one can consider first the quotient $F(C)/\langle \sigma\rangle$. The fixed locus of $\sigma$ on $F(C)$ is a smooth surface (see \cite[Example~6.4]{BCS}), its image  determines the singular locus on  $F(C)/\langle \sigma\rangle$. By \cite[Example 6.4]{BCS} we know that the action on the holomorphic 2-form is  multiplication by $\zeta_3$. Using the formula of \S\ref{sssec:nonsymplectic quotients}, the action in local coordinates at a fixed point is $\diag(1,\zeta_3,1,\zeta_3)$. The age of $\sigma$ at a singular point is $2/3$, and the age of $\sigma^2$ is  $4/3$.  This means that the singularities are strictly klt. We consider now the quotient $F(C)/\langle\iota\circ\sigma\rangle$ which has klt\, singularities by \cite[Proposition 5.20]{KM} and we study the fixed locus of $\iota\circ\sigma$. On $C$ this is the point $\vartheta\coloneqq(1:0:0:0:0:0)$ and the cubic surface $\cK_3\coloneqq\{x_0=x_5=0, G(x_1,\ldots,x_4)=0\}$. Similar methods as in~\cite[Example~6.4]{BCS} show that  the fixed locus on $F(C)$ is the union of $27$ points corresponding to the lines on $\cK_3$ and a curve $T$ parametrising the lines of $C$ through~$\vartheta$ and cutting~$\cK_3$. Since $(\iota\circ\sigma)^2=\sigma^2$ the latter fixed locus is contained in the $\sigma$-fixed surface on $F(C)$. Comparing with the local action of $\sigma$ and putting $\zeta_6:=-\zeta_3$, we conclude that at any point of the curve~$T$, the action of $\iota\circ\sigma$ can be diagonalized as $\diag(1,\zeta_6,\zeta_6^3,\zeta_6^4)$,  and at an isolated fixed point the local action is $\diag(\zeta_6^3,\zeta_6^4, \zeta_6^3,\zeta_6^4)$. Then one computes the ages of all the powers of~$\iota\circ\sigma$ and one sees that in both cases, the quotient has strictly klt singularities. 
        Finally we get that $F(C)/\langle \sigma\rangle$ and $X$ are uniruled log-Enriques varieties.

	\subsubsection{Index $6$ log-Enriques varieties}\label{ss:index6}
	
	Let $A = E\times E_\omega$, where $E$ is an elliptic curve  and $E_\omega = \frac{\IC}{\IZ\oplus \omega \IZ}$ where $\omega$ is a primitive third root of unity. Take $u\in E$, $v\in E_\omega$ and consider the order $6$ automorphism $f_A(x, y) = (-x+u, \omega y+ v)$. Assuming that $u$ and $v$ are $(n+1)$-torsion points, the natural automorphism $f_A^{[n+1]}$ of $\Hilb {n+1} A$ respects the fiber $\Kum n A$ and, as observed at the beginning of Section~\ref{s:examples_nonsingular}, it acts purely nonsymplectically.  By Proposition~\ref{prop:conseq1}, we get that the $2n$-dimensional quotient $\Kum n A /\langle\Kum n {f_A}\rangle$ has trivial canonical bundle if $n \equiv 0 \mod 6$, but it is a log-Enriques variety of IHS type otherwise:
	\begin{itemize}
		\item of index $2$ if $n \equiv 3\mod 6$;
		\item of index~$3$ if $n \equiv \pm 2 \mod 6$;
		\item of index~$6$ if $n \equiv \pm 1 \mod 6$.
	\end{itemize}
	
	\subsubsection{Index $12$ log-Enriques varieties} \label{ss:index12}
	
	Let $A = E_\ii\times E_\omega$, where $\omega$ is a primitive $3$rd root of unity, $E_\ii = \frac{\IC}{\IZ\oplus \ii \IZ}$ and $E_\omega = \frac{\IC}{\IZ\oplus \omega \IZ}$. Take $u\in E_\ii$, $v\in E_\omega$ and consider the order $12$ automorphism $f_A(x, y) = (\ii x+u, \omega y+ v)$. Assuming that $u$ and $v$ are $(n+1)$-torsion points, the natural automorphism $f_A^{[n+1]}$ of $\Hilb {n+1} A$ respects the fiber $\Kum n A$ and, as observed at the beginning of Section~\ref{s:examples_nonsingular}, it acts purely nonsymplectically.  By Proposition~\ref{prop:conseq1}, we get that the $2n$-dimensional quotient $\Kum n A /\langle\Kum n {f_A}\rangle$ has trivial canonical bundle if $n \equiv 0 \mod 12$, but it is a log-Enriques variety of IHS type otherwise:
	\begin{itemize}
		\item of index $2$ if $n \equiv 6 \mod 12$;
		\item of index~$3$ if $n \equiv \pm 4 \mod 12$;
		\item of index~$4$ if $n \equiv \pm 3 \mod 12$;
		\item of index~$6$ if $n \equiv \pm 2 \mod 12$;
		\item of index~$12$ if $n \equiv \pm 1 \mod 12$ or  $n \equiv \pm 5 \mod 12$.
	\end{itemize}

	\subsection{Log-Enriques varieties of ISV and PSV type}\label{s:examples_singular}
	
	Although the theory of automorphism groups of irreducible symplectic varieties is not yet as well-developed as in the smooth case, we want to list here some examples of log-Enriques varieties which are obtained as quasi-\'etale cyclic quotients of an ISV or a PSV. 
	
	\subsubsection{Index $2$ log-Enriques varieties from moduli spaces of semistable sheaves on $K3$ surfaces}\label{sssec:moduli-stable-sheaves-singular} Here, we observe that Oguiso--Schr\"oer's construction of Enriques manifolds illustrated in \S \ref{sssec:moduli-stable-sheaves} generalizes to singular moduli spaces of semistable sheaves on $K3$ surfaces, which are known to be ISV~\cite{PR}. Indeed, let $S$ be a very general $K3$ surface carrying an Enriques involution $\iota_S$. By genericity we can assume that $\rho(S)=10$ and that the Picard lattice $\Pic(S)$ is $\iota_S$-invariant. Given a non-primitive Mukai vector $w=(r,l,\chi-r)\in H^*(S,\IZ)$ such that $w^2\equiv 0 \mod 4$, $w^2\geq 4$ and a $w$-general polarization $H\in\Pic(S)$, the induced involution $\iota_S^*$ acts on the moduli space $M_{w,H}(S)$ of $H$-semistable sheaves on $S$ with Mukai vector $w$, and the quotient $M_{w,H}(S)/\langle\iota_S^*\rangle$ is a  log-Enriques variety of dimension $w^2+2$, of ISV type and index two, since $\iota_S^*$ acts freely in codimension one by Lemma \ref{lem:isotropic-fixed-locus}. 
	
	Moreover, if $\chi$ is odd, $\iota_S^*$ acts freely on the regular locus of $M_{w,H}(S)$, which is the locus of $H$-stable sheaves on $S$. In order to show this, it is enough to observe that the proof of ~\cite[Theorem~5.3]{OS1} works: if a stable sheaf $\cF$ is fixed, by descent there exists a coherent sheaf $\cF'$ on $S/\langle\iota_S\rangle$ such that $\cF=p^*\cF'$, and thus $\chi=\chi(\cF)=2\chi(\cF')$, in contradiction with $\chi$ being odd.
	
	If $w=kv$ with $k\in\IZ$ and $v\in H^*(S,\IZ)$ primitive, the condition that $\chi$ is odd implies that $k$ is odd, so in particular the freeness of the action on the regular locus of $M_{w,H}(S)$ does not occur in the case that $M_{w,H}(S)$ is birational to an $OG10$ manifold.
	
	\subsubsection{Index $2$ log-Enriques varieties from relative Prym varieties} Let $S$ be a very general $K3$ surface carrying an Enriques involution $i_S$, as in \S\ref{sssec:moduli-stable-sheaves} and \S\ref{sssec:moduli-stable-sheaves-singular}. Let $\pi\colon S\to T:=S/\langle i_S\rangle$ be the quotient Enriques surface, let $C$ be a smooth curve of genus $g$ on $T$ with primitive class in $\NS(T)$ and denote  $D:=\pi^{-1}(C)$. The moduli space $M_{v,D}(S)$ is singular when $v=(0,D,2-2g)\in H^*(S,\IZ)$ (see ~\cite{ASF}) and it is an ISV (see~\cite{BCGPSV_Prym} and references therein). Moreover, since $D^2=4g-4$, it is of dimension $2(2g-1)$, hence it follows from \S\ref{sssec:moduli-stable-sheaves-singular} that $i^*_S$ acts freely on $(M_{v,D}(S))_\reg$. On the other hand, the moduli space is also endowed with the duality $j$, which is a nonsymplectic involution commuting with~$i^*_S$. The connected component $\mathcal{P}_{v,D}$ of the zero-section 
	inside the fixed locus of the symplectic involution $i^*_S\circ j$ is the so-called relative Prym variety, and it is an ISV by~\cite{ASF}. The nonsymplectic involution $i^*_S$ restricts to $\mathcal{P}_{v,D}$ and acts freely on $\mathcal{P}_{v,D}\cap (M_{v,D}(S))_\reg$. For all genera $g\geq 2$ the quotient $\mathcal{P}_{v,D}/\langle i^*_S\rangle$ is a log-Enriques variety of ISV type and of index two. 
	
	This construction also works when $S$ is a very general $K3$ surface endowed with a nonsymplectic involution $i$ such that $\NS(S)\simeq U(2)\oplus E_8(-1)^{\oplus 2}$ or $\NS(S)\simeq U\oplus E_8(-1)\oplus D_4(-1)^{\oplus 2}$. In such a case, the quotient surface $T$ is rational. By \cite{BCGPSV_Prym}, one can also construct relative Prym varieties $\cP_{v,D}$ starting from a curve $C$ on $T$ of genus $g$ and considering the moduli space $M_{v,D}(S)$ of $D$-semistable sheaves on  $S$ with Mukai vector $(0,D,1-g(D))$, where $D:=\pi^{-1}(C)$ is of genus $g(D)$. Moreover, in this case $S$ also carries an Enriques involution $i_S$ by \cite{Ohashi}, which commutes with~$i$. As a consequence, whenever $D$ is also invariant for $i_S$, the involution $i_S^*$ acts on $\cP_{v,D}$ and the quotient  $\mathcal{P}_{v,D}/\langle i^*_S\rangle$ is a log-Enriques variety of index two, which is of PSV type (\resp of ISV type) when $C$ and $D$ satisfy the assumptions of \cite[Theorem 1.2]{BCGPSV_Prym} (\resp of \cite[Theorem 1.3 and 1.4]{BCGPSV_Prym}).

	\bibliographystyle{amsplain}
	\bibliography{Bib-Enriques}

@article {ADH,
	AUTHOR = {Arapura, D. and Dimca, A. and Hain, R.},
	TITLE = {On the fundamental groups of normal varieties},
	JOURNAL = {Commun. Contemp. Math.},
	FJOURNAL = {Communications in Contemporary Mathematics},
	VOLUME = {18},
	YEAR = {2016},
	NUMBER = {4},
	PAGES = {1550065, 17}}

@article {ASF,
	AUTHOR = {Arbarello, E. and Sacc\`a, G. and Ferretti, A.},
	TITLE = {Relative {P}rym varieties associated to the double cover of an
	{E}nriques surface},
	JOURNAL = {J. Differential Geom.},
	FJOURNAL = {Journal of Differential Geometry},
	VOLUME = {100},
	YEAR = {2015},
	NUMBER = {2},
	PAGES = {191--250},
	ISSN = {0022-040X,1945-743X}}

@article {BGL,
	AUTHOR = {Bakker, B. and Guenancia, H. and Lehn, C.},
	TITLE = {Algebraic approximation and the decomposition theorem for
	{K}\"{a}hler {C}alabi-{Y}au varieties},
	JOURNAL = {Invent. Math.},
	FJOURNAL = {Inventiones Mathematicae},
	VOLUME = {228},
	YEAR = {2022},
	NUMBER = {3},
	PAGES = {1255--1308}}

@book {BHPV,
	AUTHOR = {Barth, W. P. and Hulek, K. and Peters, C. A. M. and
	Van de Ven, A.},
	TITLE = {Compact complex surfaces},
	SERIES = {Ergebnisse der Mathematik und ihrer Grenzgebiete. 3. Folge. A
	Series of Modern Surveys in Mathematics [Results in
	Mathematics and Related Areas. 3rd Series. A Series of Modern
	Surveys in Mathematics]},
	VOLUME = {4},
	EDITION = {Second},
	PUBLISHER = {Springer-Verlag, Berlin},
	YEAR = {2004}}

@book {Beauville_book,
	AUTHOR = {Beauville, A.},
	TITLE = {Complex algebraic surfaces},
	SERIES = {London Mathematical Society Student Texts},
	VOLUME = {34},
	EDITION = {Second},
	NOTE = {Translated from the 1978 French original by R. Barlow, with
	assistance from N. I. Shepherd-Barron and M. Reid},
	PUBLISHER = {Cambridge University Press, Cambridge},
	YEAR = {1996}}

@article{Beauville_remarks,
author = {Beauville, A.},
 title = {Some remarks on {K{\"a}hler} manifolds with {{\(c_ 1=0\)}}},
 year = {1983},
 language = {English},
 howpublished = {Classification of algebraic and analytic manifolds, {Proc}. {Symp}., {Katata}/{Jap}. 1982, {Prog}. {Math}. 39, 1-26 (1983).}}

@article {Beauville_invol,
	AUTHOR = {Beauville, A.},
	TITLE = {Antisymplectic involutions of holomorphic symplectic
	manifolds},
	JOURNAL = {J. Topol.},
	FJOURNAL = {Journal of Topology},
	VOLUME = {4},
	YEAR = {2011},
	NUMBER = {2},
	PAGES = {300--304},
	ISSN = {1753-8416,1753-8424}}

@article {Beauville_c1nul,
	AUTHOR = {Beauville, A.},
	TITLE = {Vari\'{e}t\'{e}s {K}\"{a}hleriennes dont la premi\`ere classe de {C}hern est
	nulle},
	JOURNAL = {J. Differential Geom.},
	FJOURNAL = {Journal of Differential Geometry},
	VOLUME = {18},
	YEAR = {1983},
	NUMBER = {4},
	PAGES = {755--782 (1984)}}

@article {Beauville_symplectic,
	AUTHOR = {Beauville, A.},
	TITLE = {Symplectic singularities},
	JOURNAL = {Invent. Math.},
	FJOURNAL = {Inventiones Mathematicae},
	VOLUME = {139},
	YEAR = {2000},
	NUMBER = {3},
	PAGES = {541--549}}

@article{BGMM,
	title      = {Terminalizations of quotients of compact hyperk{\"a}hler manifolds by induced symplectic automorphisms},
    author     = {Bertini, V. and Grossi, A. and Mauri, M. and Mazzon, E.},
    url        = {https://epiga.episciences.org/13054},
    doi        = {10.46298/epiga.2025.13054},
    journal    = {EPIGA},
    volume     = {9},
    eid        = {14},
    pages      = {1-53},
    year       = {2025}}

@misc{BGS,
	title={Sur le produit de vari\'et\'es localement factorielles ou Q-factorielles},
	author={Boissi{\`e}re, S. and Gabber, O. and Serman, O.},
	note={arXiv:1104.1861},
	year={2011}}

@article {BL_global,
	AUTHOR = {Bakker, B. and Lehn, C.},
	TITLE = {A global {T}orelli theorem for singular symplectic varieties},
	JOURNAL = {J. Eur. Math. Soc. (JEMS)},
	FJOURNAL = {Journal of the European Mathematical Society (JEMS)},
	VOLUME = {23},
	YEAR = {2021},
	NUMBER = {3},
	PAGES = {949--994}}

@article {BL_moduli,
	AUTHOR = {Bakker, B. and Lehn, C.},
	TITLE = {The global moduli theory of symplectic varieties},
	JOURNAL = {J. Reine Angew. Math.},
	FJOURNAL = {Journal f\"{u}r die Reine und Angewandte Mathematik. [Crelle's
	Journal]},
	VOLUME = {790},
	YEAR = {2022},
	PAGES = {223--265}}

@article {BG_OG10,
    AUTHOR = {Billi, S. and Grossi, A.},
     TITLE = {Non-symplectic automorphisms of prime order of {O}'{G}rady's
              tenfolds and cubic fourfolds},
   JOURNAL = {Int. Math. Res. Not. IMRN},
  FJOURNAL = {International Mathematics Research Notices. IMRN},
      YEAR = {2025},
    NUMBER = {12},
     PAGES = {Paper No. rnaf159, 34}
}

@misc{BGGG_OG10,
	title={Non-existence of Enriques manifolds from OG10 type manifolds},
	author={Billi, S.  and Giovenzana, F. and Giovenzana, L.  and Grossi, A.},
	note={arXiv:2501.06893  },
	year={2025}}

@article {BCS,
	AUTHOR = {Boissi\`ere, S. and Camere, C. and Sarti, A.},
	TITLE = {Classification of automorphisms on a deformation family of
	hyper-{K}\"ahler four-folds by {$p$}-elementary lattices},
	JOURNAL = {Kyoto J. Math.},
	FJOURNAL = {Kyoto Journal of Mathematics},
	VOLUME = {56},
	YEAR = {2016},
	NUMBER = {3},
	PAGES = {465--499}
}

@article{BCMS,
	AUTHOR = {Boissi\`ere, S. and Camere, C. and Mongardi, G.
	and Sarti, A.},
	TITLE = {Isometries of ideal lattices and hyperk\"ahler manifolds},
	JOURNAL = {Int. Math. Res. Not. IMRN},
	FJOURNAL = {International Mathematics Research Notices. IMRN},
	YEAR = {2016},
	NUMBER = {4},
	PAGES = {963--977}
}

@article {BNWS,	
	AUTHOR = {Boissi\`ere, S. and Nieper-Wi{\ss}kirchen, M. and Sarti,
	A.},
	TITLE = {Higher dimensional {E}nriques varieties and automorphisms of
	generalized {K}ummer varieties},
	JOURNAL = {J. Math. Pures Appl. (9)},
	FJOURNAL = {Journal de Math\'{e}matiques Pures et Appliqu\'{e}es. Neuvi\`eme S\'{e}rie},
	VOLUME = {95},
	YEAR = {2011},
	NUMBER = {5},
	PAGES = {553--563}}

@incollection {BNWT,
	AUTHOR = {Boissi\`ere, S. and Nieper-Wi{\ss} kirchen, M. and Tari, K.},
	TITLE = {Some algebraic and geometric properties of hyper-{K}\"{a}hler
	fourfold symmetries},
	BOOKTITLE = {Geometry at the frontier---symmetries and moduli spaces of
	algebraic varieties},
	SERIES = {Contemp. Math.},
	VOLUME = {766},
	PAGES = {57--70},
	PUBLISHER = {Amer. Math. Soc., [Providence], RI},
	YEAR = {2021}}

@article{BCGPSV_Prym,
	title={Irreducible symplectic varieties via relative {P}rym varieties}, 
	author={Brakkee, E. and Camere, C. and Grossi, A. and  Pertusi, L. and  Saccà, G. and Viktorova, S.},
	journal = {Adv. Math.},
	volume = {490},
    pages = {110826},
	year = {2026}}

@article{C,
	AUTHOR = {Camere, C.},
	TITLE = {Symplectic involutions of holomorphic symplectic four-folds},
	JOURNAL = {Bull. Lond. Math. Soc.},
	FJOURNAL = {Bulletin of the London Mathematical Society},
	VOLUME = {44},
	YEAR = {2012},
	NUMBER = {4},
	PAGES = {687--702},
}

@article{CC,
	AUTHOR = {Camere, C. and Cattaneo, Al.},
	TITLE = {Non-symplectic automorphisms of odd prime order on manifolds of {$K3^{[n]}$}-type},
	JOURNAL = {Manuscripta Math.},
	FJOURNAL = {Manuscripta Mathematica},
	VOLUME = {163},
	YEAR = {2020},
	NUMBER = {3-4},
	PAGES = {299--342}}

@article {CCC,
	AUTHOR = {Camere, C. and Cattaneo, Al. and Cattaneo, An.},
	TITLE = {Non-symplectic involutions on manifolds of {$K3^{[n]}$}-type},
	JOURNAL = {Nagoya Math. J.},
	FJOURNAL = {Nagoya Mathematical Journal},
	VOLUME = {243},
	YEAR = {2021},
	PAGES = {278--302}}

@article {BC,
	AUTHOR = {Brandhorst, S. and Cattaneo, A.},
	TITLE = {Prime order isometries of unimodular lattices and
	automorphisms of {IHS} manifolds},
	JOURNAL = {Int. Math. Res. Not. IMRN},
	FJOURNAL = {International Mathematics Research Notices. IMRN},
	YEAR = {2023},
	NUMBER = {18},
	PAGES = {15584--15638}}

@article {CGKK,
	AUTHOR = {Camere, C. and Garbagnati, A. and Kapustka, G.
	and Kapustka, M.},
	TITLE = {Projective orbifolds of {N}ikulin type},
	JOURNAL = {Algebra Number Theory},
	FJOURNAL = {Algebra \& Number Theory},
	VOLUME = {18},
	YEAR = {2024},
	NUMBER = {1},
	PAGES = {165--208}}

@article {CGM,
	AUTHOR = {Camere, C. and Garbagnati, A. and Mongardi, G.},
	TITLE = {Calabi-{Y}au quotients of hyperk\"{a}hler fourfolds},
	JOURNAL = {Canad. J. Math.},
	FJOURNAL = {Canadian Journal of Mathematics. Journal Canadien de
	Math\'{e}matiques},
	VOLUME = {71},
	YEAR = {2019},
	NUMBER = {1},
	PAGES = {45--92}}

@book {CossecDolgachev,
	AUTHOR = {Cossec, F. R. and Dolgachev, I. V.},
	TITLE = {Enriques surfaces. {I}},
	SERIES = {Progress in Mathematics},
	VOLUME = {76},
	PUBLISHER = {Birkh\"{a}user Boston, Inc., Boston, MA},
	YEAR = {1989}}

@article{Demleitner,
	author = {Demleitner, A.},
	title = {Classification of {Bagnera{\textendash}de} {Franchis} {Varieties} in {Small} {Dimensions}},
	journal = {Annales de la Facult\'e des sciences de Toulouse : Math\'ematiques},
	pages = {111--133},
	publisher = {Universit\'e Paul Sabatier, Toulouse},
	volume = {Ser. 6, 29},
	number = {1},
	year = {2020}}

@article {Druel,
	AUTHOR = {Druel, S.},
	TITLE = {A decomposition theorem for singular spaces with trivial
	canonical class of dimension at most five},
	JOURNAL = {Invent. Math.},
	FJOURNAL = {Invent. Math.},
	VOLUME = {211},
	YEAR = {2018},
	NUMBER = {1},
	PAGES = {245--296}}

@article{DRTX,
	title={{MMP} for {E}nriques pairs and singular {E}nriques varieties}, 
	author={Denisi, F. A.  and Ríos Ortiz, Á. D.  and Tsakanikas, N.  and Xie, Z. },
	journal = {J. Éc. polytech. Math.},
	volume = {13},
	year = {2026},
	pages =  {629–686}}

@incollection {Dolgachev,
	AUTHOR = {Dolgachev, I. V.},
	TITLE = {A brief introduction to {E}nriques surfaces},
	BOOKTITLE = {Development of moduli theory---{K}yoto 2013},
	SERIES = {Adv. Stud. Pure Math.},
	VOLUME = {69},
	PAGES = {1--32},
	PUBLISHER = {Math. Soc. Japan, [Tokyo]},
	YEAR = {2016},
	ISBN = {978-4-86497-032-7}}

@incollection {Fujiki,
	AUTHOR = {Fujiki, A.},
	TITLE = {On primitively symplectic compact {K}\"{a}hler {$V$}-manifolds
	of dimension four},
	BOOKTITLE = {Classification of algebraic and analytic manifolds ({K}atata,
	1982)},
	SERIES = {Progr. Math.},
	VOLUME = {39},
	PAGES = {71--250},
	PUBLISHER = {Birkh\"{a}user Boston, Boston, MA},
	YEAR = {1983},
	ISBN = {0-8176-3137-2}}

@incollection {FultonLazarsfeld,
	AUTHOR = {Fulton, W. and Lazarsfeld, R.},
	TITLE = {Connectivity and its applications in algebraic geometry},
	BOOKTITLE = {Algebraic geometry ({C}hicago, {I}ll., 1980)},
	SERIES = {Lecture Notes in Math.},
	VOLUME = {862},
	PAGES = {26--92},
	PUBLISHER = {Springer, Berlin-New York},
	YEAR = {1981}}

@misc{Gachet,
	title={Well-clipped cones behave themselves under all finite quotients, the cone conjecture under most}, 
	author={Gachet, C.},
	year={2025},
	note={ArXiv:2504.01753} 
}

@article {GGK,
	AUTHOR = {Greb, D. and Guenancia, H. and Kebekus, S.},
	TITLE = {Klt varieties with trivial canonical class: holonomy,
	differential forms, and fundamental groups},
	JOURNAL = {Geom. Topol.},
	FJOURNAL = {Geometry \& Topology},
	VOLUME = {23},
	YEAR = {2019},
	NUMBER = {4},
	PAGES = {2051--2124}}

@article {GKKP,
	AUTHOR = {Greb, D. and Kebekus, S. and Kov\'{a}cs, S. J. and
	Peternell, T.},
	TITLE = {Differential forms on log canonical spaces},
	JOURNAL = {Publ. Math. Inst. Hautes \'{E}tudes Sci.},
	FJOURNAL = {Publications Math\'{e}matiques. Institut de Hautes \'{E}tudes
	Scientifiques},
	NUMBER = {114},
	YEAR = {2011},
	PAGES = {87--169}}

@article {GKP_etale,
	AUTHOR = {Greb, D. and Kebekus, S. and Peternell, T.},
	TITLE = {\'{E}tale fundamental groups of {K}awamata log terminal
	spaces, flat sheaves, and quotients of abelian varieties},
	JOURNAL = {Duke Math. J.},
	FJOURNAL = {Duke Mathematical Journal},
	VOLUME = {165},
	YEAR = {2016},
	NUMBER = {10},
	PAGES = {1965--2004}}

@incollection {GKP_singular,
	AUTHOR = {Greb, D. and Kebekus, S. and Peternell, T.},
	TITLE = {Singular spaces with trivial canonical class},
	BOOKTITLE = {Minimal models and extremal rays ({K}yoto, 2011)},
	SERIES = {Adv. Stud. Pure Math.},
	VOLUME = {70},
	PAGES = {67--113},
	PUBLISHER = {Math. Soc. Japan, [Tokyo]},
	YEAR = {2016}}

@article {GrossiOG6,
	AUTHOR = {Grossi, A.},
	TITLE = {Nonsymplectic automorphisms of prime order on {O}'{G}rady's
	sixfolds},
	JOURNAL = {Rev. Mat. Iberoam.},
	FJOURNAL = {Revista Matem\'{a}tica Iberoamericana},
	VOLUME = {38},
	YEAR = {2022},
	NUMBER = {4},
	PAGES = {1199--1218}
}

@article {Gro,
	AUTHOR = {Grothendieck, A.},
	TITLE = {Sur quelques points d'alg\`ebre homologique},
	JOURNAL = {Tohoku Math. J. (2)},
	FJOURNAL = {The Tohoku Mathematical Journal. Second Series},
	VOLUME = {9},
	YEAR = {1957},
	PAGES = {119--221}}

@article {Hartshorne,
	AUTHOR = {Hartshorne, R.},
	TITLE = {Stable reflexive sheaves},
	JOURNAL = {Math. Ann.},
	FJOURNAL = {Mathematische Annalen},
	VOLUME = {254},
	YEAR = {1980},
	NUMBER = {2},
	PAGES = {121--176}}

@article {HoringPeternell,
	AUTHOR = {H\"{o}ring, A. and Peternell, T.},
	TITLE = {Algebraic integrability of foliations with numerically trivial
	canonical bundle},
	JOURNAL = {Invent. Math.},
	FJOURNAL = {Inventiones Mathematicae},
	VOLUME = {216},
	YEAR = {2019},
	NUMBER = {2},
	PAGES = {395--419},
	ISSN = {0020-9910,1432-1297}}

@article {Kahn,
	AUTHOR = {Kahn, B.},
	TITLE = {Sur le groupe des classes d'un sch\'{e}ma arithm\'{e}tique},
	NOTE = {With an appendix by Marc Hindry},
	JOURNAL = {Bull. Soc. Math. France},
	FJOURNAL = {Bulletin de la Soci\'{e}t\'{e} Math\'{e}matique de France},
	VOLUME = {134},
	YEAR = {2006},
	NUMBER = {3},
	PAGES = {395--415}}

@article {Kawamata,
	AUTHOR = {Kawamata, Y.},
	TITLE = {Minimal models and the {K}odaira dimension of algebraic fiber
	spaces},
	JOURNAL = {J. Reine Angew. Math.},
	FJOURNAL = {Journal f\"{u}r die Reine und Angewandte Mathematik. [Crelle's
	Journal]},
	VOLUME = {363},
	YEAR = {1985},
	PAGES = {1--46}}

@book {Kollar,
    AUTHOR = {Koll\'ar, J.},
     TITLE = {Singularities of the minimal model program},
    SERIES = {Cambridge Tracts in Mathematics},
    VOLUME = {200},
      NOTE = {With a collaboration of S\'andor Kov\'acs},
 PUBLISHER = {Cambridge University Press, Cambridge},
      YEAR = {2013},
     PAGES = {x+370},
      ISBN = {978-1-107-03534-8},
   
}

@book {KM,
	AUTHOR = {Koll\'{a}r, J. and Mori, S.},
	TITLE = {Birational geometry of algebraic varieties},
	SERIES = {Cambridge Tracts in Mathematics},
	VOLUME = {134},
	NOTE = {With the collaboration of C. H. Clemens and A. Corti,
	Translated from the 1998 Japanese original},
	PUBLISHER = {Cambridge University Press, Cambridge},
	YEAR = {1998},
	PAGES = {viii+254}}

@article{Kebekus,
	title={Pull-back morphisms for reflexive differential forms},
	author={Kebekus, Stefan},
	journal={Adv. Math.},
	volume={245},
	pages={78--112},
	year={2013}
}

@article {KS,
	AUTHOR = {Kebekus, S. and Schnell, C.},
	TITLE = {Extending holomorphic forms from the regular locus of a
	complex space to a resolution of singularities},
	JOURNAL = {J. Amer. Math. Soc.},
	FJOURNAL = {Journal of the American Mathematical Society},
	VOLUME = {34},
	YEAR = {2021},
	NUMBER = {2},
	PAGES = {315--368}}

@book {Lang,
	AUTHOR = {Lang, S.},
	TITLE = {Abelian varieties},
	NOTE = {Reprint of the 1959 original},
	PUBLISHER = {Springer-Verlag, New York-Berlin},
	YEAR = {1983},
	PAGES = {xii+256}}

@article {Matsumura-Oort,
	AUTHOR = {Matsumura, H. and Oort, F.},
	TITLE = {Representability of group functors, and automorphisms of
	algebraic schemes},
	JOURNAL = {Invent. Math.},
	FJOURNAL = {Inventiones Mathematicae},
	VOLUME = {4},
	YEAR = {1967},
	PAGES = {1--25},
	ISSN = {0020-9910}
}

@article {MTW,
	AUTHOR = {Mongardi, G. and Tari, K. and Wandel, M.},
	TITLE = {Prime order automorphisms of generalised {K}ummer fourfolds},
	JOURNAL = {Manuscripta Math.},
	FJOURNAL = {Manuscripta Mathematica},
	VOLUME = {155},
	YEAR = {2018},
	NUMBER = {3-4},
	PAGES = {449--469}
}

@incollection {Mukai,
	AUTHOR = {Mukai, S.},
	TITLE = {Lecture notes on {$K3$} and {E}nriques surfaces},
	BOOKTITLE = {Contributions to algebraic geometry},
	SERIES = {EMS Ser. Congr. Rep.},
	PAGES = {389--405},
	PUBLISHER = {Eur. Math. Soc., Z\"{u}rich},
	YEAR = {2012},
	ISBN = {978-3-03719-114-9}}

@article {Nagata,
	AUTHOR = {Nagata, M.},
	TITLE = {On the purity of branch loci in regular local rings},
	JOURNAL = {Illinois J. Math.},
	FJOURNAL = {Illinois Journal of Mathematics},
	VOLUME = {3},
	YEAR = {1959},
	PAGES = {328--333},
	ISSN = {0019-2082}}

@book {Nakayama,
    AUTHOR = {Nakayama, N.},
     TITLE = {Zariski-decomposition and abundance},
    SERIES = {MSJ Memoirs},
    VOLUME = {14},
 PUBLISHER = {Mathematical Society of Japan, Tokyo},
      YEAR = {2004},
     PAGES = {xiv+277}
}

@article {Nakayama-Zhang,
	AUTHOR = {Nakayama, N. and Zhang, D.-Q.},
	TITLE = {Polarized endomorphisms of complex normal varieties},
	JOURNAL = {Math. Ann.},
	FJOURNAL = {Mathematische Annalen},
	VOLUME = {346},
	YEAR = {2010},
	NUMBER = {4},
	PAGES = {991--1018}
}

@article {OS1,
	AUTHOR = {Oguiso, K. and Schr\"{o}er, S.},
	TITLE = {Enriques manifolds},
	JOURNAL = {J. Reine Angew. Math.},
	FJOURNAL = {Journal f\"{u}r die Reine und Angewandte Mathematik. [Crelle's
	Journal]},
	VOLUME = {661},
	YEAR = {2011},
	PAGES = {215--235}}

@article {Ohashi,
	AUTHOR = {Ohashi, H.},
	TITLE = {On the number of {E}nriques quotients of a {$K3$} surface},
	JOURNAL = {Publ. Res. Inst. Math. Sci.},
	FJOURNAL = {Kyoto University. Research Institute for Mathematical
	Sciences. Publications},
	VOLUME = {43},
	YEAR = {2007},
	NUMBER = {1},
	PAGES = {181--200}
}

@misc{PS,
	title={On the cone conjecture for {E}nriques manifolds}, 
	author={Pacienza, G. and Sarti, A.},
	year={2023},
	note={ArXiv:2303.07095}}

@incollection {Perego_examples,
	AUTHOR = {Perego, A.},
	TITLE = {Examples of irreducible symplectic varieties},
	BOOKTITLE = {Birational geometry and moduli spaces},
	SERIES = {Springer INdAM Ser.},
	VOLUME = {39},
	PUBLISHER = {Springer, Cham},
	YEAR = {2020},
	PAGES = {151--172}}

@article {PR,
	AUTHOR = {Perego, A. and Rapagnetta, A.},
	TITLE = {Irreducible symplectic varieties from moduli spaces of sheaves
	on {K}3 and {A}belian surfaces},
	JOURNAL = {Algebr. Geom.},
	FJOURNAL = {Algebraic Geometry},
	VOLUME = {10},
	YEAR = {2023},
	NUMBER = {3},
	PAGES = {348--393}
}

@incollection {ReidYPG,
	AUTHOR = {Reid, M.},
	TITLE = {Young person's guide to canonical singularities},
	BOOKTITLE = {Algebraic geometry, {B}owdoin, 1985 ({B}runswick, {M}aine,
	1985)},
	SERIES = {Proc. Sympos. Pure Math.},
	VOLUME = {46, Part 1},
	PAGES = {345--414},
	PUBLISHER = {Amer. Math. Soc., Providence, RI},
	YEAR = {1987}}

@incollection {Reid_canonical,
	AUTHOR = {Reid, M.},
	TITLE = {Canonical {$3$}-folds},
	BOOKTITLE = {Journ\'{e}es de {G}\'{e}ometrie {A}lg\'{e}brique d'{A}ngers, {J}uillet
	1979/{A}lgebraic {G}eometry, {A}ngers, 1979},
	PAGES = {273--310},
	PUBLISHER = {Sijthoff \& Noordhoff, Alphen aan den Rijn---Germantown, Md.},
	YEAR = {1980}}

@article {Schwald,
	Author = {Schwald, M.},
	Title = {Fujiki relations and fibrations of irreducible symplectic varieties},
	FJournal = {{\'E}pijournal de G{\'e}om{\'e}trie Alg{\'e}brique. EPIGA},
	Journal = {{\'E}pijournal de G{\'e}om. Alg{\'e}br., EPIGA},
	ISSN = {2491-6765},
	Volume = {4},
	Pages = {19},
	Note = {Id/No 7},
	Year = {2020}}

@misc{Yoshikawa,
	AUTHOR = {Yoshikawa, K.-I.},
	TITLE =  {Enriques 2n-folds and analytic torsion},
	NOTE = {in preparation}}

@article {Zariski,
	AUTHOR = {Zariski, O.},
	TITLE = {On the purity of the branch locus of algebraic functions},
	JOURNAL = {Proc. Nat. Acad. Sci. U.S.A.},
	FJOURNAL = {Proceedings of the National Academy of Sciences of the United
	States of America},
	VOLUME = {44},
	YEAR = {1958},
	PAGES = {791--796},
	ISSN = {0027-8424}}

@article {Zhang,
	AUTHOR = {Zhang, D.-Q.},
	TITLE = {Logarithmic {E}nriques surfaces},
	JOURNAL = {J. Math. Kyoto Univ.},
	FJOURNAL = {Journal of Mathematics of Kyoto University},
	VOLUME = {31},
	YEAR = {1991},
	NUMBER = {2},
	PAGES = {419--466}}

@article{Zhang2,
author = {Zhang, D.-Q.},
title = {{Logarithmic Enriques surfaces, II}},
volume = {33},
journal = {Journal of Mathematics of Kyoto University},
number = {2},
publisher = {Duke University Press},
pages = {357 -- 397},
year = {1993},
doi = {10.1215/kjm/1250519265}}

@misc{zong,
	title={Almost non-degenerate abelian fibrations}, 
	author={Zong, Y.},
	year={2016},
	note={ArXiv:1406.5956}}
	
\end{document}